\documentclass[12pt]{article}
\usepackage{amsmath,amssymb,amsthm,amscd,dsfont}
\numberwithin{equation}{section} \allowdisplaybreaks
\begin{document}
\newtheorem{theorem}{Theorem}[section]
\newtheorem{defin}{Definition}[section] \newtheorem{prop}{Proposition}[section]
\newtheorem{corol}{Corollary}[section]
\newtheorem{lemma}{Lemma}[section]
\newtheorem{rem}{Remark}[section]
\newtheorem{example}{Example}[section]
\title{On the geometry of double field theory} \author{{\small
by}\vspace{2mm}\\Izu Vaisman} \date{}
\maketitle
{\def\thefootnote{*}}
\begin{center}
\begin{minipage}{12cm} A{\footnotesize BSTRACT. Double field theory was developed by theoretical physicists as a way to encompass $T$-duality.
In this paper, we express the basic notions of the theory in differential-geometric invariant terms, in the framework of para-K\"ahler manifolds. We define metric algebroids, which are vector bundles with a bracket of cross sections that has the same metric compatibility property as a Courant bracket. We show that a double field gives rise to two canonical connections, whose scalar curvatures can be integrated to obtain actions. Finally, in analogy with Dirac structures, we define and study para-Dirac structures on double manifolds.} \end{minipage} \end{center} \vspace*{5mm} 
\section{Introduction}
Double field theory is a way to express string theory that encompasses $T$-duality and it was intensively studied in the theoretical physics literature of the last
years (see \cite{{HHZ},{HK},{HZ},{JLP},{JLP1}} and the references therein). In particular, relations of the theory with Hitchin's generalized geometry \cite{H} have been noticed (e.g., \cite{HZ}).

The aim of this paper is to formulate some of the geometry of double field theory in differential-geometric, invariant terms. We explain that doubled space-time should be seen as a flat, bi-Lagrangian (equivalently, para-K\"ahler) manifold. On the tangent bundle of such a manifold, we define a {\it metric algebroid structure}, with a Courant-like bracket (the C-bracket of the physics literature, e.g., \cite{HZ}). The corresponding non-skew-symmetric product defines a generalized Lie derivative (gauge transformations, in the language of physics).

Then, we express the equivalence between the field and a generalized metric, which was previously introduced in the physics literature, e.g., \cite{HHZ}, in geometric terms. Furthermore, we obtain canonical, generalized-metric connections and the global expression of the corresponding covariant derivative. We also define a generalized curvature tensor and a corresponding scalar curvature that may be used in the construction of an action of the field. Here and also later on, the word ``generalized" alludes  to generalized geometry \cite{{G},{G2},{H}}, etc.
Finally, we discuss Dirac-like structures in the metric algebroid of a flat, para-K\"ahler manifold.
\section{Metric algebroids and brackets}
Double field theory adds to the coordinates $x^i$ $(i=1,...,m)$ of the space-time manifold an equal number
of new coordinates $\tilde{x}_i$, thus defining a manifold $ M^{2m}$
called the double of the original manifold. (In some versions of the theory the number of added coordinates is smaller, but, we will ignore that.)
Moreover, there is an implicit understanding
that $\partial/\partial x^i$ have a
``covariant behavior" while $\partial/\partial\tilde{x}_i$ have a ``contravariant behavior". The aim of this procedure is to obtain a manifold
whose structure group can be reduced to the group $O(m,m)$, which is required for $T$-duality. The coordinates $(x^i,\tilde{x}_j)$ will be called {\it distinguished local coordinates} and $M$ also has arbitrary local coordinates defined by a differentiable transformation $$y^u=y^u(x^i,\tilde{x}_j)\hspace{2mm}(u=1,...,2m).$$

The distinguished coordinate transformations will be of the form
\begin{equation}\label{transfcoord} x^i=x^i(x^{'j}),\;
\tilde{x}_i=\tilde{x}_i(\tilde{x}'_k),\end{equation}
where
\begin{equation}\label{condtransf}
\frac{\partial\tilde{x}_i}{\partial\tilde{x}'_j}
=\frac{\partial x^{'j}}{\partial x^i}.\end{equation}
Since we have different sets of variables in the two sides, the partial
derivatives of (\ref{transfcoord}) must be locally constant, hence, $M$
is a particular type of a locally affine manifold, with affine coordinate
transformations of the local form
\begin{equation}\label{locafin}
x^i=\alpha^i_jx'^j+\alpha^i_0,
\; \tilde{x}_i=\beta_i^k\tilde{x}'_k+\beta_i^0\end{equation}
where $\alpha^i_j,\alpha^i_0,\beta_i^k,\beta_i^0$ are constants and $\alpha^i_j\beta_i^k=\delta_j^k$, the Kronecker index (we use the Einstein summation convention).

Condition (\ref{condtransf}) implies that the coordinate transformations
(\ref{transfcoord}) preserve the closed non degenerate $2$-form\footnote{For wedge products, we will use Cartan's evaluation conventions, e.g., $(\alpha\wedge\beta)(X,Y)=\alpha(X)\beta(Y)-\alpha(Y)\beta(X)$.}
\begin{equation}\label{expromega} \omega=dx^i\wedge d\tilde{x}_i.\end{equation}
Thus,  $M$ is endowed with
a symplectic form $\omega$ and with two Lagrangian foliations $L,\tilde{L}$
defined by $\tilde{x}_i=const.$, $x^i=const.$, respectively\footnote{By $L,\tilde{L}$ we denote both the foliations and the tangent bundles of the leaves.}. In the geometric
literature, such manifolds are called bi-Lagrangian or para-K\"ahler manifolds
\cite{AMT}, which are {\it flat} if $\omega$ has the expression (\ref{expromega}).

The reason for the last name is the equivalence of the symplectic form with a metric defined as follows. The two Lagrangian foliations may be seen as integrating the $\pm1$-eigenbundles of the para-complex structure $F:TM\rightarrow TM$ defined by
$$ F\frac{\partial}{\partial x^i}
=\frac{\partial}{\partial x^i},\,  F\frac{\partial}{\partial \tilde{x}_i}
=-\frac{\partial}{\partial \tilde{x}_i},
$$
with the characteristic properties
$$F^2=Id,\;\omega(FX,FY)=-\omega(X,Y),\;\;X,Y\in TM.$$
Then, the required metric is
\begin{equation}\label{defgamma} \gamma(X,Y)=\omega(FX,Y)=
dx^i\otimes d\tilde{x}_i.\end{equation}
Conversely, we have
$$ \omega(X,Y)=\gamma(FX,Y).$$

The expression (\ref{defgamma}) shows that the metric $\gamma$ is
non-degenerate and neutral, $L,\tilde{L}$ are maximal $\gamma$-isotropic bundles and
$$
\gamma(\frac{\partial}{\partial x^i},\frac{\partial}{\partial \tilde{x}_j})=\delta_i^j.$$

Notice that the dual bundles are given by
$$L^*=ann\,\tilde{L},\,\tilde{L}^*=ann\,L,$$ where $ann$ denotes the annihilator of a vector space or bundle.
The metric $\gamma$ produces the isomorphisms $\varphi:\tilde{L}\rightarrow L^*$ where $\varphi(\tilde{X})=\flat_\gamma\tilde{X}$ ($\tilde{X}\in\tilde{L}$) and a corresponding isomorphism $ \tilde{\varphi}:L\oplus\tilde{L}\rightarrow TM$ given by
$$\begin{array}{l}
\tilde{\varphi}(X,\alpha)=X+\sharp_\gamma\alpha\;\;(X\in L,\alpha\in L^*),\vspace*{2mm}\\
\tilde{\varphi}^{-1}Z=\frac{1}{2}((Id+F)Z,\flat_\gamma(Id-F)Z)\;\; (Z\in TM).\end{array}$$ In particular, $\tilde{\varphi}(0,dx^i)=
\partial/\partial\tilde{x}_i$.

The isomorphism $\tilde{\varphi}$ allows us to transfer the metric $\gamma$ and the form $\omega$ from $TM$ to $L\oplus L^*$ and the results, which we keep denoting by $\gamma,\omega$, are
\begin{equation}\label{gammabyL} \begin{array}{c} \gamma((X,\alpha),(Y,\beta))=\beta(X)+\alpha(Y),\;
 \omega((X,\alpha),(Y,\beta))=\beta(X)-\alpha(Y)\vspace*{2mm}\\ (X,Y\in L,\alpha,\beta\in L^*=ann\,\tilde{L}).\end{array}\end{equation}
$\gamma$ is the metric used in generalized geometry\footnote{In many papers on generalized geometry the metric and form are (\ref{gammabyL}) multiplied by $(1/2)$ but some papers use exactly (\ref{gammabyL}), e.g., \cite{G2}.}.\vspace*{2mm}\\
{\bf Notation Convention.} Based on the above, in the remainder of the paper we identify $L\oplus\tilde{L}$ with $TM$ via $\tilde{\varphi}$ (which is an isometry for the metric $\gamma$) and use the following notation convention: the symbols $L\oplus\tilde{L}$ and $TM$ will be used interchangeably and seen as synonymous; symbols like $(X,\alpha)$ ($X\in L,\alpha\in L^*=ann\,\tilde{L}$) and $X+\sharp_\gamma\alpha$ will be used interchangeably, seen as synonymous and also denoted by a boldface character, e.g., $\mathbf{X}$.\\

We assume that the reader knows the basic notions of generalized geometry such as Lie and Courant algebroids, and the corresponding calculus \cite{{LWX},{Mk}}, which serve as background to our study. Double field theory requires related, but different, notions \cite{{HK},{HZ}}. The need for such notions comes from the fact that the Lie bracket of vector fields is not adequate for $T$-duality.
\begin{defin}\label{defalgL} {\rm Let $E\rightarrow M$ be a vector bundle  endowed with a symmetric, non degenerate, inner product\footnote{$\Gamma$ denotes the space of cross sections of a bundle and $\odot$ denotes the symmetric tensor product.}
$g\in\Gamma\odot^2E^*$ (a {\it metric}) and a morphism $\rho:E\rightarrow TM$ (an {\it anchor}); then, we also have the morphism $\partial=(1/2)\sharp_g\hspace*{1pt}^t\hspace*{-1pt}\rho:T^*M\rightarrow E$, where $t$ denotes transposition. Assume that there exists an $\mathds{R}$-bilinear product $\bigstar:\Gamma E\times\Gamma E\rightarrow\Gamma E$ that satisfies the following properties:
$$\begin{array}{l}
1)\;(\rho e)(g(e_1,e_2))=g(e\bigstar e_1,e_2)+g(e_1,e\bigstar e_2),\vspace*{2mm}\\
2)\;e\bigstar e=\partial(g(e,e))\hspace{2mm}(\partial f=\partial(df),\,\forall f\in C^\infty(M)).\end{array}$$
Then, the quadruple $(E,g,\rho,\bigstar)$ is called a {\it metric algebroid}. Axiom 1) is the $g$-{\it compatibility} axiom and axiom 2) is the {\it normalization} axiom.}\end {defin}

We notice the following equivalent form of the definition of $\partial$
\begin{equation}\label{defpartial2} g(\partial f,e)=\frac{1}{2}(\rho e)f.
\end{equation} On the other hand we can prove
\begin{prop}\label{nonflin} The product of a metric algebroid satisfies the properties
$$\begin{array}{l}
a)\;e_1\bigstar(fe_2)=f(e_1\bigstar e_2)+((\rho e_1)f)e_2,\vspace*{2mm}\\
b)\;(fe_1)\bigstar e_2=f(e_1\bigstar e_2)-((\rho e_2)f)e_1 +2g(e_1,e_2)\partial f.\end{array}$$
\end{prop}
\begin{proof} The $g$-compatibility condition for the triple $(e_1,fe_2,e)$ gives $$(\rho e_1)(g(fe_2,e))=g(e_1\bigstar(fe_2),e)+g(fe_2,e_1\bigstar e)$$ and also $$(\rho e_1)(g(fe_2,e))=(\rho e_1)(f)g(e_2,e)+f(\rho e_1)(g(e_2,e)).$$ By equating the results, and since $g$ is non degenerate, we get property a).

Then, by polarization, the normalization axiom 2) gives
\begin{equation}\label{normal2} e_1\bigstar e_2+e_2\bigstar e_1=2\partial(g(e_1,e_2))\end{equation}
and property b) is a direct consequence of a) and (\ref{normal2}).\end{proof}
\begin{rem}\label{linax1} {\rm Formula (\ref{defpartial2}) and properties a), b) imply that axiom 1) is invariant under the multiplication of any of the arguments $e,e_1,e_2$ by a function $f\in C^\infty(M)$.}\end{rem}
\begin{prop}\label{algcucr} A triple $(E,g,\rho)$ as above is a metric algebroid iff there exists an $\mathds{R}$-bilinear, skew symmetric bracket
$[\,,\,]:\Gamma E\times\Gamma E\rightarrow\Gamma E$ that satisfies the axiom
\begin{equation}\label{ax1'}
(\rho e)(g(e_1,e_2))=g([e,e_1]+\partial(g(e,e_1)),e_2)+g(e_1,[e,e_2]
+\partial(g(e,e_2))).\end{equation}\end{prop}
\begin{proof} The product and the bracket reciprocally define each other by the relation
\begin{equation}\label{cr*} [e_1,e_2]=e_1\bigstar e_2-\partial(g(e_1,e_2)),
\end{equation} which shows the equivalence of the skew symmetry of the bracket with the normalization axiom 2) of the product.
Furthermore, (\ref{cr*}) ensures the equivalence between the $g$-compatibility axiom 1) and (\ref{ax1'}).
\end{proof}

Notice that properties a), b) are equivalent with
\begin{equation}\label{lincroset} [e_1,fe_2]=f[e_1,e_2]+(\rho e_1)(f)e_2
-g(e_1,e_2)\partial f,\end{equation} which may also be deduced from (\ref{ax1'}), directly, as in the proof of a).
\begin{example}\label{exma} {\rm Any Courant algebroid (e.g., \cite{{LWX},{V5}}) is a metric algebroid.}\end{example}
\begin{prop}\label{existcr} {\rm\cite{V5}} Any metric vector bundle $(E,g)$ endowed with an anchor $\rho:E\rightarrow TM$ has infinitely many structures of a metric algebroid. Namely, any $g$-preserving connection on $E$ defines a metric bracket and leads to a bijective correspondence between the metric brackets on $(E,g,\rho)$ and the $3$-forms $B\in\Gamma\wedge^3E^*$.\end{prop}
\begin{proof}
Fix a $g$-preserving connection $\nabla^0$ on $E$, i.e.,
\begin{equation}\label{preservg}
X(g(e_1,e_2))=g(\nabla^0_Xe_1,e_2)+g(e_1,\nabla^0_Xe_2),\hspace{3mm}X\in TM.
\end{equation}
Define a skew symmetric product\footnote{There is no relationship between $\wedge_{\nabla^0}$ (or similar products) and the usual wedge product.} $e_1\wedge_{\nabla^0} e_2\in\Gamma E$ by
\begin{equation}\label{wedgenabla}
g(e,e_1\wedge_{\nabla^0} e_2)=\frac{1}{2}[g(e_1,\nabla^0_{\rho e}e_2)
-g(e_2,\nabla^0_{\rho e}e_1)].\end{equation}
This product satisfies the property
\begin{equation}\label{propcdot} e_1\wedge_{\nabla^0} (fe_2)
=f(e_1\wedge_{\nabla^0} e_2)+g(e_1,e_2)\partial f.\end{equation}
Then, put
\begin{equation}\label{crptD}
[e_1,e_2]_{\nabla^0}=\nabla^0_{\rho e_1}e_2-\nabla^0_{\rho e_2}e_1
-e_1\wedge_{\nabla^0} e_2.\end{equation}
A technical computation shows that the skew symmetric bracket (\ref{crptD})  satisfies the axiom (\ref{ax1'}); we will say that (\ref{crptD}) is the {\it metric $\nabla^0$-bracket}.

Furthermore, any other metric bracket on $(E,\rho,g)$ must have the form
$$ [e_1,e_2]'=[e_1,e_2]_{\nabla^0}+\beta(e_1,e_2). $$
Put $$B(e_1,e_2,e_3)=g(\beta(e_1,e_2),e_3).$$
$B$ obviously is skew symmetric in $(e_1,e_2)$ and the skew symmetry in $(e_2,e_3)$ follows by subtracting the expressions of (\ref{ax1'}) for the two brackets. In a similar way, (\ref{lincroset}) implies the $C^\infty(M)$-linearity of $B$ in each of its arguments. Therefore, $B$ is a $3$-form on the vector bundle $E$, i.e., $B\in\Gamma\wedge^3L^*$.
\end{proof}

The following general considerations concerning $g$-preserving connections $\nabla$ on a metric algebroid $(E,g,\rho,[\,,\,])$ will be needed later.
With $\nabla$, we define the {\it modified bracket}
\begin{equation}\label{crmodD} [e_1,e_2]^{\nabla}=[e_1,e_2]+e_1\wedge_{\nabla} e_2,
\end{equation} where $e_1\wedge_{\nabla} e_2$ is defined by (\ref{wedgenabla}) with $\nabla$ instead of $\nabla^0$.
$e_1\wedge_{\nabla} e_2$ is skew symmetric and satisfies the condition
\begin{equation}\label{crmodDf} [e_1,fe_2]^{\nabla}=f[e_1,e_2]^{\nabla}+(\rho e_1)(f)e_2.
\end{equation}

The connection $\nabla$ has the {\it modified torsion tensor}
\begin{equation}\label{modtors0} T^{\nabla}(e_1,e_2)=
\nabla_{\rho e_1}e_2-\nabla_{\rho e_2}e_1-[e_1,e_2]^{\nabla}
\end{equation}
and the {\it Gualtieri torsion} \cite{G2}
\begin{equation}\label{defGualt0}
\mathcal{T}^{\nabla}(e_1,e_2,e_3)=g(T^{\nabla}(e_1,e_2),e_3).\end{equation}
The tensorial character of $T^{\nabla}$ and $ \mathcal{T}^{\nabla}$ follows from (\ref{crmodDf}). The fact that $B$ of Proposition \ref{existcr} is a form implies that the Gualtieri torsion is totally skew symmetric.

We may also define the {\it modified curvature operator}:
\begin{equation}\label{modcurvD} R^{\nabla}(e_1,e_2)e_3=\nabla_{\rho e_1}\nabla_{\rho e_2}e_3
-\nabla_{\rho e_2}\nabla_{\rho e_1}e_3-\nabla_{\rho[e_1,e_2]^{\nabla}}e_3,\end{equation}
which is easily seen to be $C^\infty(M)$-trilinear.

Below, we will show that, on a flat bi-Lagrangian manifold $(M,L,\tilde{L})$, $TM=L\oplus\tilde{L}$ has a natural structure of a metric algebroid with metric $\gamma$ and anchor $\rho=Id$ and with a bracket that is the relevant bracket for $T$-duality because, as we will see later on, it puts $L$ and $L^*$ on an equal footing. We will also see that this bracket is the $C$-bracket of the physics literature \cite{{HHZ},{HZ}}.

The manifold $M$ has the canonical flat connection $\nabla^0$ defined by the local equations
\begin{equation}\label{eqDM} \nabla^0\frac{\partial}{\partial x^i}=0,\,
\nabla^0\frac{\partial}{\partial \tilde{x}_j}=0\end{equation}
with respect to the distinguished coordinates $(x^i,\tilde{x}_j)$ of (\ref{locafin}) and $\nabla^0$ coincides with both the Levi-Civita connection of $\gamma$ and the symplectic, bi-Lagrangian connection of $\omega$.

Accordingly, we have the metric $\nabla^0$-bracket given by (\ref{crptD}), which, in the present case and since $\nabla^0$ has a vanishing usual torsion, becomes
\begin{equation}\label{Ccroset} \begin{array}{c}
[\mathbf{X},\mathbf{Y}]_{\nabla^0}=\nabla^0_{\mathbf{X}}\mathbf{Y}
-\nabla^0_{\mathbf{Y}}\mathbf{X}
-\mathbf{X}\wedge_{\nabla^0} \mathbf{Y} =[ \mathbf{X},\mathbf{Y}]-\mathbf{X}\wedge_{\nabla^0} \mathbf{Y}\vspace*{2mm}\\ (\gamma(\mathbf{Z},\mathbf{X}\wedge_{\nabla^0} \mathbf{Y})
=\frac{1}{2}[\gamma(\mathbf{X},\nabla^0_\mathbf{Z}\mathbf{Y})-
\gamma(\mathbf{Y},\nabla^0_\mathbf{Z}\mathbf{X})]) \end{array}\end{equation}
where the unindexed bracket is the usual Lie bracket of vector fields. If either $ \mathbf{X},\mathbf{Y}\in L$ or $ \mathbf{X},\mathbf{Y}\in \tilde{L}$ then $ \mathbf{X}\wedge_{\nabla^0}\mathbf{Y}=0$ and the $\nabla^0$-bracket reduces to the Lie bracket.
By (\ref{cr*}), the metric product that corresponds to the $\nabla^0$-bracket is
\begin{equation}\label{prodD} \mathbf{X}\bigstar_{\nabla^0}\mathbf{Y}=[\mathbf{X},\mathbf{Y}]-\mathbf{X}\barwedge _{\nabla^0} \mathbf{Y}+\partial(\gamma(\mathbf{X},\mathbf{Y})).\end{equation}
For the $\nabla^0$-bracket, equation (\ref{ax1'}) takes the form
\begin{equation}\label{axvCalg}\begin{array}{c} \mathbf{Z}(\gamma( \mathbf{X},\mathbf{Y}))
=\gamma([\mathbf{Z},\mathbf{X}]_{\nabla^0},\mathbf{Y})+\frac{1}{2} \mathbf{Y}(\gamma( \mathbf{Z},\mathbf{X}))\vspace*{2mm}\\ +\gamma([\mathbf{Z},\mathbf{Y}]_{\nabla^0},\mathbf{X})+\frac{1}{2} \mathbf{X}(\gamma( \mathbf{Z},\mathbf{Y})).\end{array}\end{equation}
\begin{prop}\label{prop0} The metric bracket and product defined by the connection $\nabla^0$ are characterized by the distinguished local coordinate conditions
\begin{equation}\label{starcu0} \begin{array}{l}
\frac{\partial}{\partial x^i}\wedge_{\nabla^0}
\frac{\partial}{\partial x^j}=0,\; \frac{\partial}{\partial x^i}\wedge_{\nabla^0}
\frac{\partial}{\partial \tilde{x}_j}=0,\;
\frac{\partial}{\partial \tilde{x}_i}\wedge_{\nabla^0}
\frac{\partial}{\partial \tilde{x}_j}=0,
\vspace*{2mm}\\

[\frac{\partial}{\partial x^i},
\frac{\partial}{\partial x^j}]_{\nabla^0}=0,\; [\frac{\partial}{\partial x^i},
\frac{\partial}{\partial \tilde{x}_j}]_{\nabla^0}=0,\;
[\frac{\partial}{\partial \tilde{x}_i},
\frac{\partial}{\partial \tilde{x}_j}]_{\nabla^0}=0,\vspace*{2mm}\\

\frac{\partial}{\partial x^i}\bigstar_{\nabla^0}
\frac{\partial}{\partial x^j}=0,\; \frac{\partial}{\partial x^i}\bigstar_{\nabla^0}
\frac{\partial}{\partial \tilde{x}_j}=0,\;
\frac{\partial}{\partial \tilde{x}_i}\bigstar_{\nabla^0}
\frac{\partial}{\partial \tilde{x}_j}=0.
\end{array}
\end{equation} \end{prop}
\begin{proof}
The first line of (\ref{starcu0}) follows from (\ref{eqDM}). Then, the second line is a direct consequence of (\ref{Ccroset}). Finally, the third line follows from the formula
\begin{equation}\label{partial3} \partial f=\frac{1}{2}(0,df),\;\;\forall f\in C^\infty(M),\end{equation}
which is a consequence of (\ref{defpartial2}).

The $\nabla^0$-bracket is the only metric bracket of $(TM,\gamma,Id)$ that satisfies (\ref{starcu0}) because there exists only one extension of (\ref{starcu0}) that satisfies (\ref{lincroset}).
\end{proof}

If we move from $TM$ to the synonymous bundle $L\oplus L^*$, the brackets (\ref{starcu0}) yield corresponding brackets
\begin{equation}\label{CLcu0} \begin{array}{c}[(\frac{\partial}{\partial x^i},0),
(\frac{\partial}{\partial x^j},0)]_{\nabla^0}=0,\;
[(\frac{\partial}{\partial x^i},0),
(0,dx^j)]_{\nabla^0}=0,
\vspace*{2mm}\\

[(0,dx^i),(0,dx^j)]_{\nabla^0}=0.\end{array}\end{equation}
With (\ref{lincroset}), the brackets (\ref{CLcu0}) extend to the metric bracket that consists of the Lie algebroid brackets of $L,L^*$ and of
$$\begin{array}{l}
[(X,0),(0,\beta)]_{\nabla^0}
=([-\frac{1}{2}\beta_j\frac{\partial\xi^h}{\partial\tilde{x}_j}
-\frac{1}{4}\xi^j\frac{\partial\beta_j}{\partial\tilde{x}_h}
+\frac{1}{4}\beta_j\frac{\partial\xi^j}{\partial\tilde{x}_h}]
\frac{\partial}{\partial x^h},\vspace*{2mm}\\

[\xi^j\frac{\partial\beta_h}{\partial x^j}-\frac{1}{2}\xi^j\frac{\partial\beta_j}{\partial x^h}
+\frac{1}{2}\beta_j\frac{\partial\xi^j}{\partial x^h}]dx^h)
\;\;(X=\xi^j\frac{\partial}{\partial x^j}\in L,\beta=\beta_jdx^j\in L^*).
\end{array}$$

The corresponding general expression of this bracket is similar to the expression of the bracket on the double of a Lie bialgebroid \cite{LWX}:
\begin{equation}\label{CcuLWZ} \begin{array}{c}
[(X,\alpha),(Y,\beta)]_{\nabla^0}=
([X,Y]+\mathcal{L}^*_\alpha Y-\mathcal{L}^*_\beta X-\frac{1}{2}d^*(\alpha(Y)-
\beta(X)),\vspace*{2mm}\\
\frac{1}{2}\flat_\gamma[\sharp_\gamma\alpha,\sharp_\gamma\beta]
+\mathcal{L}_X\beta-\mathcal{L}_Y\alpha+\frac{1}{2}d(\alpha(Y)-\beta(X))),
\end{array}\end{equation}
where $X,Y\in L,\alpha,\beta\in L^*$, $\mathcal{L},d$ are the Lie derivative and exterior differential in the Lie algebroid $L$, while $\mathcal{L}^*,d^*$ are the corresponding operators in the Lie algebroid $L^*$ with the bracket indicated in the first term of the second component of (\ref{CcuLWZ}).
This may be checked by calculations and using the definition of the Lie derivative and the exterior differential in a Lie algebroid \cite{Mk}.

Formula (\ref{CcuLWZ}) justifies the earlier assertion that the $\nabla^0$-bracket puts $L$ and $L^*$ on the same footing. In \cite{HZ}, it was shown that formula (\ref{CcuLWZ}) expresses the $C$-bracket used in the physics literature, therefore, the $\nabla^0$-bracket is the $C$-bracket.

We also notice another general expression of the $\nabla^0$-bracket that follows from results given in \cite{V3}, namely,
\begin{equation}\label{CcuLtildeL}
\begin{array}{l}
[X_L,Y_L]_{\nabla^0}=[X_L,Y_L],\, [X_{\tilde{L}},Y_{\tilde{L}}]_{\nabla^0}=[X_{\tilde{L}},Y_{\tilde{L}}],
\vspace*{2mm}\\
\gamma(Z_L,[X_L,Y_{\tilde{L}}]_{\nabla^0})=\gamma([Z_L,X_L],Y_{\tilde{L}})
+X_L(\gamma(Z_L,Y_{\tilde{L}}))
-\frac{1}{2}Z_L(\gamma(X_L,Y_{\tilde{L}})),\vspace*{2mm}\\
\gamma(Z_{\tilde{L}},[X_L,Y_{\tilde{L}}]_{\nabla^0})=-\gamma(X_L,
[Z_{\tilde{L}},Y_{\tilde{L}}])
-Y_{\tilde{L}}(\gamma(Z_{\tilde{L}},X_L))
+\frac{1}{2}Z_{\tilde{L}}(\gamma(X_L,Y_{\tilde{L}})),
\end{array}
\end{equation}
where the indices $L,\tilde{L}$ indicate the subbundles where the respective vectors sit. The easiest way to check these formulas is by checking that they imply the second line of (\ref{starcu0}) and yield a bracket that satisfies (\ref{lincroset}).

We end this section by the following remark.
\begin{rem}\label{Sfield} {\rm Like the usual Courant bracket \cite{G}, the $\nabla^0$-bracket and product admit new automorphisms called $S$-field transformations\footnote{In generalized geometry these are called $B$-field transformations, but, here, we have followed the literature on double field theory where the letter $B$ is used to denote the $2$-form of the field.}, where $S\in\wedge^2L^*$ is a closed form on $L$. These are defined by
\begin{equation}\label{Sfieldtr} (X,\alpha)\mapsto(X,\alpha+i(X)S)
\hspace{3mm}(X\in L,\alpha\in L^*).\end{equation} It suffices to check that the transformation (\ref{Sfieldtr}) preserves the brackets (\ref{CLcu0}). Only the first bracket (\ref{CLcu0}) imposes a restriction, which is exactly $d_LS=0$, where $d_L$ denotes the exterior differential of the Lie algebroid $L$.}\end{rem}
\section{Brackets of strongly foliated vector fields}
The $\bigstar_{\nabla^0}$-product satisfies the axioms 1), 2) of a metric algebroid with $E=TM,g=\gamma,\rho=Id$. Therefore, $\bigstar_{\nabla^0}$ cannot satisfy the Leibniz (Jacobi) identity
\begin{equation}\label{Leibnizrule}
\mathbf{X}\bigstar_{\nabla^0}(\mathbf{Y}\bigstar_{\nabla^0} \mathbf{Z})-(\mathbf{X}\bigstar_{\nabla^0} \mathbf{Y})\bigstar_{\nabla^0} \mathbf{Z} -\mathbf{Y}\bigstar_{\nabla^0}(\mathbf{X}\bigstar_{\nabla^0} \mathbf{Z})=0\end{equation} since otherwise $(TM,\gamma,Id,\bigstar_{\nabla^0})$ would be a Courant algebroid, and we would necessarily get the preservation of the bracket by the anchor (e.g., \cite{V5}). In our case, this means the equality of the $\nabla^0$-bracket with the Lie bracket, in contradiction with (\ref{Ccroset}).

However, the identity (\ref{Leibnizrule}) holds for foliated cross sections\footnote{We refer the reader to \cite{Mol} for the required notions of foliation theory. In particular, we recall that, on a foliated manifold, an object is said to be foliated if it is the lift of a corresponding object on the space of the leaves.} of $TM$.
More precisely, the locally affine structure whose existence follows from the coordinate transformations (\ref{locafin}) may also be seen as an $ \tilde{L}$-foliated structure\footnote{It may also be seen as an $L$-foliated structure. We make the convention to use the term foliated for $ \tilde{L}$-foliated.} on the vector bundle $TM$. A foliated cross section of the $ \tilde{L}$-foliated bundle $TM$, i.e., a vector field $ \mathbf{X}$ on $M$ of the local form
\begin{equation}\label{folvf} \mathbf{X}=\xi^i(x^j)\frac{\partial}{\partial x^i} +\tilde{\xi}_i(x^j)\frac{\partial}{\partial \tilde{x}_i}
\end{equation} with respect to distinguished coordinates,
will be called a {\it strongly foliated vector field}. We will denote by $\chi_{{\rm sf}}(M)$the space of strongly foliated vector fields.
The demand that $ \mathbf{X}$ is strongly foliated is more restrictive  than the demand that $ \mathbf{X}$ is a foliated vector field in the sense of foliation theory, where, generally, $TM$ is not a foliated bundle; for a foliated (not strongly) vector field $ \mathbf{X}$ only the components $\xi^i$ do not depend on the coordinates $ \tilde{x}_i$.

We begin by proving the following auxiliary results.
\begin{lemma}\label{corolarLz} For any $\tilde{L}$-foliated function $f$ and any $ \mathbf{Z}\in\chi_{{\rm sf}}(M)$ one has
\begin{equation}\label{dfX} (\partial f)\bigstar_{\nabla^0}\mathbf{Z}=0, \;\mathbf{Z}\bigstar_{\nabla^0}(\partial f)=\partial( \mathbf{Z}f).\end{equation}
\end{lemma}
\begin{proof} If we prove the first relation (\ref{dfX}), the second will follow from (\ref{normal2}). As for the first relation, using Proposition \ref{nonflin} and formula (\ref{partial3}), we see that, for any $\tilde{L}$-foliated functions $f,h$, one has
$$\partial f\bigstar_{\nabla^0}(h\mathbf{Z})=h(\partial f\bigstar_{\nabla^0}\mathbf{Z}).$$
Thus, it suffices to check the relation for $\mathbf{Z}=\partial/\partial x^i,\mathbf{Z}=\partial/\partial\tilde{x}_i$. This follows from (\ref{starcu0}), (\ref{partial3}) and Proposition \ref{nonflin}.
\end{proof}
\begin{lemma}\label{lemaLz} If $\mathbf{X},\mathbf{Y}\in\chi_{{\rm sf}}(M)$, then $\mathbf{X}\wedge_{\nabla^0}\mathbf{Y}\in\chi_{{\rm sf}}(M)$ and is tangent to the leaves of $\tilde{L}$. Under the same hypothesis, one also has
\begin{equation}\label{Liesf} pr_L[\mathbf{X},\mathbf{Y}]_{\nabla^0}= [pr_L\mathbf{X},pr_L\mathbf{Y}].\end{equation}\end{lemma}
\begin{proof} The first assertion follows from formulas (\ref{starcu0}) and property (\ref{propcdot}). The second assertion follows by noticing that $\mathbf{X}\wedge_{\nabla^0}\mathbf{Y}\in\chi_{{\rm sf}}(M)$ implies
$$[\mathbf{X},\mathbf{Y}]= [pr_L\mathbf{X},pr_L\mathbf{Y}]+[pr_L\mathbf{X},pr_{\tilde{L}}\mathbf{Y}]
+[pr_{\tilde{L}}\mathbf{X},pr_L\mathbf{Y}],$$ where the last two terms belong to $\tilde{L}$.\end{proof}

Now, we will prove the announced result.
\begin{prop}\label{propLz} The restriction of the product $\bigstar_{\nabla^0}$ to $\chi_{{\rm sf}}(M)$ satisfies the identity {\rm(\ref{Leibnizrule})}.\end{prop}
\begin{proof}
From Lemma \ref{lemaLz} and formula (\ref{Ccroset}), it follows that $\chi_{{\rm sf}}(M)$ is closed under the $\nabla^0$-bracket and we also have the decomposition of the $\nabla^0$-bracket of strongly foliated vector fields into the $L$ and $\tilde{L}$ components. Furthermore, since if $\mathbf{X},\mathbf{Y}\in\chi_{{\rm sf}}(M)$, $\gamma(\mathbf{X},\mathbf{Y})$ is an $\tilde{L}$-foliated function, formula (\ref{prodD}) shows that $\chi_{{\rm sf}}(M)$ is also closed under the product $\bigstar_{\nabla^0}$, and we have the decomposition of the product into the $L$ and $\tilde{L}$ components.

To prove that (\ref{Leibnizrule}) holds we notice that, if a function $f$ is $\tilde{L}$-foliated, then $\mathbf{X}f=(pr_L\mathbf{X})f$. In particular, by Lemma \ref{lemaLz}, $$\mathbf{X}\bigstar_{\nabla^0}\mathbf{Y}(f)=[pr_L\mathbf{X},pr_L\mathbf{Y}](f),
\;\;\forall \mathbf{X},\mathbf{Y}\in\chi_{{\rm sf}}(M).$$
This observation, together with Lemma \ref{corolarLz} allows us to check that the left hand side of (\ref{Leibnizrule})
is linear with respect to the multiplication of any of its arguments by an $\tilde{L}$-foliated function. Accordingly, it suffices to prove (\ref{Leibnizrule}) for the local vector fields $(\partial/\partial x^i, \partial/\partial \tilde{x}_i)$. This straightforwardly follows from (\ref{starcu0}).\end{proof}
\begin{corol}\label{corolJac} The restriction of the bracket $[\,,\,]_{\nabla^0}$ to $\chi_{{\rm sf}}(M)$ satisfies the property
\begin{equation}\label{Jptcr} \begin{array}{r} \sum_{Cycl(\mathbf{X},\mathbf{Y},\mathbf{Z})} [[\mathbf{X},\mathbf{Y}]_{\nabla^0}, \mathbf{Z}]_{\nabla^0}=
\frac{1}{3}\sum_{Cycl(\mathbf{X},\mathbf{Y},\mathbf{Z})} \partial(\gamma([\mathbf{X},\mathbf{Y}]_{\nabla^0},\mathbf{Z}))
\vspace*{2mm}\\ = \frac{1}{2}\sum_{Cycl(\mathbf{X},\mathbf{Y},\mathbf{Z})} \partial(\gamma([\mathbf{X},\mathbf{Y}],\mathbf{Z})) =\sum_{Cycl(\mathbf{X},\mathbf{Y},\mathbf{Z})} \partial(\gamma(\mathbf{Y}\wedge_{\nabla^0}\mathbf{X},\mathbf{Z})).
\end{array}\end{equation}
\end{corol}
\begin{proof} In (\ref{Leibnizrule}), express the products by brackets using the relation (\ref{cr*}). Then, compute and make reductions using the property (\ref{ax1'}). The result is
$$\begin{array}{c}\sum_{Cycl(\mathbf{X},\mathbf{Y},\mathbf{Z})} [[\mathbf{X},\mathbf{Y}]_{\nabla^0}, \mathbf{Z}]_{\nabla^0}= \partial(\gamma([\mathbf{X},\mathbf{Y}]_{\nabla^0},\mathbf{Z}))
+\frac{1}{2}\partial( \mathbf{Y}(\gamma(\mathbf{X},\mathbf{Z})))\vspace*{2mm}\\
-\frac{1}{2}\partial( \mathbf{X}(\gamma(\mathbf{Y},\mathbf{Z}))).\end{array}$$
If we add to the previous equality its two cyclic permutations, the result is the first equality in (\ref{Jptcr}). The two other equalities follow from (\ref{Ccroset}) and the remark that
\begin{equation}\label{equax}\sum_{Cycl(\mathbf{X},\mathbf{Y},\mathbf{Z})} \gamma([\mathbf{X},\mathbf{Y}],\mathbf{Z})=
2\sum_{Cycl(\mathbf{X},\mathbf{Y},\mathbf{Z})} \gamma(\mathbf{Y}\wedge_{\nabla^0}\mathbf{X},\mathbf{Z}).\end{equation} \end{proof}
\begin{corol}\label{algCfol} The triple $(TM,\gamma,pr_L)$ is an $\tilde{L}$-transversal Courant algebroid on $M$.\end{corol}
\begin{proof} The notion of a transversal Courant algebroid with respect to a foliation was defined in \cite{VMedJ}. The definition is the same as that of a Courant algebroid, but, everything is required to be foliated and the anchor goes to the transversal bundle of the foliation. The corollary follows from the formulas (\ref{ax1'}), (\ref{Ccroset}), (\ref{partial3}), (\ref{Liesf}) and Corollary \ref{corolJac}.
\end{proof}
\begin{rem}\label{obsBuc} {\rm By the extension theorem proven in \cite{VBuc}, the $\tilde{L}$-transversal Courant algebroid structure of $TM$ extends to a usual Courant algebroid structure on the vector bundle $L\oplus(ann\,L)\oplus TM\approx TM\oplus TM$, where the isomorphism is defined by sending $\lambda\in ann\,L$ to $\sharp_\gamma\lambda\in L$. The metric of the extension is $\gamma\oplus\gamma$, the anchor is
$\rho( \mathbf{X},\mathbf{Y})=pr_{\tilde{L}} \mathbf{X}+pr_L\mathbf{Y}$ and the brackets of the elements of the bases produced by distinguished coordinates are zero.}\end{rem}

Property a) of Proposition \ref{nonflin} suggests that the operation $\bigstar_{\nabla^0}$ may be used to define a {\it generalized Lie derivative} ({\it gauge transformation} e.g.,\cite{HHZ}), which we shall denote by $\mathfrak{L}_{ \mathbf{X}}$ where $ \mathbf{X}$ is any vector field on $M$. This generalized derivative acts on tensor fields as follows
\begin{equation}\label{genL} \begin{array}{l} \mathfrak{L}_{ \mathbf{X}}f= \mathbf{X}(f),\;
\mathfrak{L}_{ \mathbf{X}} \mathbf{Y}=\mathbf{X}\bigstar_{\nabla^0}\mathbf{Y},\;
(\mathfrak{L}_{ \mathbf{X}}\alpha)( \mathbf{Y})=\mathbf{X}(\alpha( \mathbf{Y}))
-\alpha( \mathfrak{L}_{\mathbf{X}} \mathbf{Y})
\vspace*{2mm}\\ (\mathfrak{L}_{ \mathbf{X}} \mathbf{T})( \mathbf{Y}_1,...,
\mathbf{Y}_p,\alpha_1,...,\alpha_q)=\mathbf{X}(\mathbf{T}( \mathbf{Y}_1,...,
\mathbf{Y}_p,\alpha_1,...,\alpha_q))
\vspace*{2mm}\\ -\sum_{i=1}^p\mathbf{T}( \mathbf{Y}_1,...,
\mathfrak{L}_{ \mathbf{X}}\mathbf{Y}_i,...,
\mathbf{Y}_p,\alpha_1,...,\alpha_q)
\vspace*{2mm}\\
-\sum_{i=1}^q\mathbf{T}( \mathbf{Y}_1,...,
\mathbf{Y}_p,\alpha_1,...,
\mathfrak{L}_{ \mathbf{X}}\alpha_i,...,\alpha_q).\end{array}\end{equation}
The result is a tensor field; it is easy to check $C^\infty(M)$-linearity with respect to all the arguments.
The last formula (\ref{genL}) is the extension of the operator $\mathfrak{L}_{ \mathbf{X}}$ given by the first three formulas to tensors, under the requirement that the extension acts on tensor products (therefore, also on exterior products) as a derivation.

The identity (\ref{Leibnizrule}) is equivalent to \begin{equation}\label{comutLiegen} \mathfrak{L}_{ \mathbf{X}}\mathfrak{L}_{ \mathbf{Y}}-\mathfrak{L}_{ \mathbf{Y}}\mathfrak{L}_{ \mathbf{X}}
=\mathfrak{L}_{ \mathbf{X}\bigstar_{\nabla^0}\mathbf{Y}},\end{equation}
where $ \mathbf{X},\mathbf{Y}$ are strongly foliated vector fields and the operators act on the space of strongly foliated vector fields. Then, if we define a strongly foliated tensor field $\mathbf{T}$ by requiring that its components with respect to distinguished coordinates be $\tilde{L}$-foliated functions, formulas (\ref{genL}) allow us to see that the condition (\ref{comutLiegen}) holds for the action on strongly foliated tensor fields.

We also notice the following local formulas obtained from (\ref{genL}), where $\mathbf{X}$ is the strongly foliated vector field given by (\ref{folvf}):
\begin{equation}\label{genLlocal} \begin{array}{l}
\mathfrak{L}_{ \mathbf{X}}(\frac{\partial}{\partial x^k})
=(\frac{\partial\tilde{\xi}_k}{\partial x^i}-
\frac{\partial\tilde{\xi}_i}{\partial x^k})\frac{\partial}{\partial \tilde{x}_i}- \frac{\partial\xi^i}{\partial x^k}\frac{\partial}{\partial x^i},\, \mathfrak{L}_{ \mathbf{X}}(\frac{\partial}{\partial \tilde{x}_k})=\frac{\partial\xi^k}{\partial x^i}\frac{\partial}{\partial \tilde{x}_i},
\vspace*{2mm}\\ \mathfrak{L}_{ \mathbf{X}}(dx^j)= \frac{\partial\xi^j}{\partial x^k}dx^k,\, \mathfrak{L}_{ \mathbf{X}}(d\tilde{x}_j)=
(\frac{\partial\tilde{\xi}_k}{\partial x^j}-
\frac{\partial\tilde{\xi}_j}{\partial x^k})dx^k- \frac{\partial\xi^k}{\partial x^j}d\tilde{x}^k.\end{array} \end{equation}
\begin{example}\label{Liegamma} {\rm
The metric $\gamma$, which by (\ref{defgamma}) is a strongly foliated tensor field, satisfies the condition $\mathfrak{L}_{ \mathbf{X}}\gamma=0$ for any vector field $ \mathbf{X}$. This is nothing but the metric axiom for the algebroid $(TM,Id,\gamma)$ written for strongly foliated arguments.}\end{example}
\begin{example}\label{LieT} {\rm If $f$ is an $\tilde{L}$-foliated function , then, $ \mathfrak{L}_{\partial f} \mathbf{T}=0$ for any strongly foliated tensor field $ \mathbf{T}$. It is a consequence of  (\ref{dfX}) and (\ref{genL}) that $ \mathfrak{L}_{\partial f} \mathbf{T}$ vanishes on strongly foliated arguments (check first for $\mathbf{T}=f,\mathbf{X},\alpha$ and then for a general $\mathbf{T}$). This suffices to justify the conclusion since it allows us to conclude that $(\mathfrak{L}_{\partial f} \mathbf{T})_x=0$ at any point $x\in M$ by extending the values of the arguments at $x$ to strongly foliated fields.}\end{example}
\section{Fields and generalized metrics}
Let $(M,L,\tilde{L})$ be a flat, bi-Lagrangian manifold. In double field theory, the initial geometric objects that define a field are $(g,B)$ where $g$ is a non degenerate metric of $L$, with positive-negative inertia indices $p,q$, $p+q=m$, and $B\in\Gamma\wedge^2L^*$ is a $2$-form. Furthermore, a dilation scalar $\phi\in C^\infty(M)$ is also required.

Since $g,B,\phi$ are geometric objects we may use arbitrary coordinates $(y^u)$ $(u=1,...,2m)$. But, physics asks for the existence of coordinates where the field only depends on the coordinates along the leaves of $L$. This is called the level matching constraint and it is equivalent with the fact that the components of $g,B$ with respect to distinguished coordinates depend only on the coordinates $x^i$. Thus, level matching constraint means that $g,B,\phi$ are foliated objects with respect to the foliation $\tilde{L}$. In invariant terms, the level matching constraint has the following expression:
$$
i(Z)\bar{g}=0,\,\mathcal{L}_Z\bar{g}=0,\,i(Z)\bar{B}=0,
\,\mathcal{L}_Z\bar{B}=0,\,Z\phi=0,\,
\forall Z\in\Gamma \tilde{L},$$ where $ \mathcal{L}$ is Lie derivative and $\bar{g},\bar{B}$ are the extensions of $g,B$ to tensor fields on $M$ that vanish if any of the arguments is in $\tilde{L}$. In other words, $\bar{g},\bar{B}$ are strongly foliated tensor fields; we may say that, generally, any strongly foliated tensor field satisfies the level matching constraint. We do not assume, a priori, that the level matching constraint holds and we will postulate this condition explicitly when used.

The pair $(g,B)$ turns out to be equivalent with a metric $\mathcal{H}$ on $TM$ that is compatible with the neutral metric $\gamma$, in the sense that it further reduces the structure group of $TM$ from $O(m,m)$ to $O(p,q)\times O(q,p)$. Such metrics will be called (improperly but briefly) {\it generalized metrics}. The metric  $\mathcal{H}$ appears in the physics literature \cite{HK}. More exactly, we have the following result.
\begin{prop}\label{thequival} There exists a bijective correspondence between the fields $(g,B)$ and the generalized metrics $\mathcal{H}$ that have a non degenerate restriction to the bundle $L^*$.\end{prop}
\begin{proof} We start with $(g,B)$ and define a corresponding $\mathcal{H}$ following \cite{G}.
The formulas
\begin{equation}\label{injections} X\mapsto(X,\flat_{B+g}X),\,
X\mapsto(X,\flat_{B-g}X),\hspace{2mm}X\in L\end{equation} (where
$\flat_g,\flat_B:L\rightarrow L^*$ are defined as in classical Riemannian geometry and the elements of $L^*$ are
identified with $1$-forms on $M$ that vanish on $\tilde{L}$) define injections
$\iota_\pm:L\rightarrow L\oplus L^*$. Since $\flat_gX=0$
only for $X=0$, the images $S_\pm$ of $\iota_\pm$ have intersection zero and
their sum is $2m$-dimensional, i.e.,
$$ S_+\oplus S_-=L\oplus L^*.$$
It follows easily that the decomposition along $S_+$ and $S_-$ is given by
\begin{equation}\label{descSpm} \begin{array}{l}(X,\alpha)=
(X_1,\flat_{B+g}X_1)+(X_2,\flat_{B-g}X_2),\vspace*{2mm}\\
X_1=\frac{1}{2}[X-\sharp_g(\flat_BX-\alpha)],\,X_2=
\frac{1}{2}[X+\sharp_g(\flat_BX-\alpha)].\end{array}
\end{equation}

The required metric $ \mathcal{H}$ is defined on $L\oplus L^*$ by
\begin{equation}\label{defHS}
\mathcal{H}|_{S_+}=\gamma|_{S^+}=2\iota_+^{-1*}g,\,\mathcal{H}|_{S_-}=-\gamma|_{S^-}
=2\iota_-^{-1*}g,\,S_+\perp_\mathcal{H}S_-\end{equation} where the star denotes the metrics induced by $g$ in $S_\pm$.

Furthermore, by our notation convention, $S_\pm$ may also be seen as subbundles of $TM$ and $\mathcal{H}$ may also be seen as a metric on $TM$, using the transfer by $\tilde{\varphi}$ defined in Section 1. The corresponding reduction of the structure group of $TM$ characterizes a generalized metric.

Using (\ref{descSpm}), we get the following general expression of $\mathcal{H}$ on $L\oplus L^*$
\begin{equation}\label{Hdesc} \mathcal{H}((X,\alpha),(Y,\beta))
=g(X,Y)+g^{-1}(\flat_BX-\alpha,\flat_BY-\beta),
\end{equation} where $X,Y\in L,\alpha,\beta\in ann\,\tilde{L}$. Using (\ref{Hdesc}) and transferring to $TM$, we obtain
\begin{equation}\label{compH} \begin{array}{c}
\mathcal{H}(\frac{\partial}{\partial x^i},\frac{\partial}{\partial x^j})
=g_{ij}-B_{ik}g^{kl}B_{lj},\vspace*{2mm}\\ \mathcal{H}(\frac{\partial}{\partial x^i},\frac{\partial}{\partial \tilde{x}_j})=
g^{jk}B_{ki},\, \mathcal{H}(\frac{\partial}{\partial \tilde{x}_i},
\frac{\partial}{\partial \tilde{x}_j})=g^{ij},\end{array}\end{equation}
where $(x^i,\tilde{x}_j)$ are distinguished, local coordinates and $B=(1/2)B_{ij}dx^i\wedge dx^j$.

A comparison with the physics literature (e.g., \cite{HHZ}) shows that $\mathcal{H}$ is the metric defined by physicists. From the last formula (\ref{compH}) we see that the restriction of $\mathcal{H}$ to $L^*$ is non degenerate.

We will also need the isomorphism $\Phi=\sharp_\mathcal{H}\circ\flat_\gamma$ of $L\oplus L^*$,
which may be defined equivalently on $TM$ by
\begin{equation}\label{Hgamma} \mathcal{H}(\Phi Z,U)=\gamma(Z,U),\hspace{3mm}
Z,U\in TM.\end{equation}
From (\ref{defHS}), we get $\flat_\mathcal{H}|_{S_+}=\flat_\gamma|_{S_+},
\flat_\mathcal{H}|_{S_-}=-\flat_\gamma|_{S_-}$, and the definition
of $\Phi$ yields $\Phi|_{S_+}=Id,\Phi|_{S_-}=-Id$. Thus, $\Phi^2=Id$
and $\Phi$ is an almost product, $\mathcal{H}$-orthogonal structure on $L\oplus L^*$, equivalently, on $TM$\footnote{The
projections of $TM$ onto the $\pm1$-eigenbundles of $\Phi$, i.e., the morphisms $(1/2)(Id\pm\Phi)$ are the projections used in \cite{{JLP},{JLP1}}.}. From (\ref{Hgamma}) we also get
\begin{equation}\label{gammaH} \mathcal{H}(Z,U)=\gamma(\Phi Z,U).\end{equation}

Now, we can show the converse. If we have a generalized, $L^*$-non degenerate metric $ \mathcal{H}$ of $M$, the reduction of the structure group of the bundle of the canonical frames of $\gamma$ from $O(m,m)$ to $O(p,q)\times O(q,p)$ produces a decomposition $L\oplus L^*=S_+\oplus S_-$ such that
$$
\mathcal{H}|_{S_+}=\gamma|_{S^+},\,\mathcal{H}|_{S_-}=-\gamma|_{S^-},
\,S_+\perp_\mathcal{H}S_-,\,S_+\perp_\gamma S_-.$$
This implies that $\Phi=\sharp_\mathcal{H}\circ\flat_\gamma$ has $S_\pm$ as its $\pm1$-eigenbundles. Thus, we see that $\Phi$ is an almost product, $\mathcal{H}$-orthogonal and $\gamma$-orthogonal structure on $L\oplus L^*$ satisfying (\ref{Hgamma}), (\ref{gammaH}), which also implies the compatibility relations
$$ \gamma(\Phi Z,\Phi U)=\gamma(Z,U),\;
\mathcal{H}(\Phi Z,\Phi U)=\mathcal{H}(Z,U).$$

Then, we shall continue as follows \cite{V4}. Using the compatibility conditions above, it follows that $\Phi$ has a matrix representation
\begin{equation}\label{matriceaPhi} \Phi\left(\begin{array}{c}
X\vspace{2mm}\\ \alpha\end{array}\right)=
\left(\begin{array}{cc}\psi&\sharp_g\vspace*{2mm}\\
\flat_{\tilde{g}}&^t\hspace{-1pt}\psi\end{array}\right)
\left(\begin{array}{c} X\vspace{2mm}\\ \alpha\end{array}\right),
\end{equation} where $\psi\in End(L)$, $\tilde{g}$ is a symmetric $2$-covariant tensor on $L$ and $g$ is a symmetric $2$-contravariant tensor on $L$, which are defined by
$$\begin{array}{c}
\mathcal{H}((X,0),(0,\beta))=\beta(\psi X),\,
\mathcal{H}((X,0),(Y,0))=\tilde{g}(X,Y),\vspace*{2mm}\\
\mathcal{H}((0,\alpha),(0,\beta))=g(\alpha,\beta),\end{array}$$
with $X,Y\in L,\alpha,\beta\in L^*$. In particular, since $\mathcal{H}$ was assumed non degenerate on $L^*$, $g$ is non degenerate and we may follow the Riemannian convention of denoting the inverse of $\sharp_g$ by $\flat_g$. Finally, the product condition $\Phi^2=Id$ is equivalent to
\begin{equation}\label{Phiprodus} \begin{array}{c}
\psi^2=Id-\sharp_g\circ\flat_{\tilde{g}},\;
\psi\circ\sharp_g+\sharp_g\circ^t\hspace{-1pt}\psi=0,\,
\flat_{\tilde{g}}\circ\psi+^t\hspace{-1pt}\psi\circ\flat_{\tilde{g}}.
\end{array}\end{equation}

Thus, the representation (\ref{matriceaPhi}) yields a metric $g$ on $L$. It also yields a $2$-form $B$ defined by
\begin{equation}\label{B} \flat_B=-\flat_g\circ\psi
\end{equation} (the skew symmetry of $B$ is a consequence of (\ref{Phiprodus})).

Furthermore, with $(g,B)$, we can define the injections $\iota_\pm:L\rightarrow L\oplus L^*$ of (\ref{injections}). Using (\ref{Phiprodus}), (\ref{matriceaPhi}) and (\ref{B}) a straightforward calculation gives
$$\Phi(X,\flat_{B\pm g}X)=\pm(X,\flat_{B\pm g}X),$$
which proves that the metric $\mathcal{H}$ is exactly the metric associated to the pair $(g,B)$ in the first part of the proof.
\end{proof}

We end this section with a few more remarks. A vector field $ \mathbf{X}$ on $M$ will be called a {\it generalized-Killing vector field} if $ \mathfrak{L}_{ \mathbf{X}} \mathcal{H}=0$.
\begin{prop}\label{Killing} If the field $(g,B)$ satisfies the level matching constraint, the strongly foliated vector field $ \mathbf{X}$ is a generalized-Killing vector field iff
$$ \mathcal{L}_{pr_L\mathbf{X}}g=0,\,
\mathcal{L}_{pr_L\mathbf{X}}B=0,$$ where $ \mathcal{L}$ is the Lie derivative of the Lie algebroid $L$.\end{prop}
\begin{proof} Since $ \mathfrak{L}_{ \mathbf{X}} \gamma=0$ (see Example \ref{Liegamma}), $ \mathfrak{L}_{ \mathbf{X}} \mathcal{H}=0$ iff $ \mathfrak{L}_{ \mathbf{X}} \Phi=0$, equivalently, iff $ \mathfrak{L}_{ \mathbf{X}} \mathbf{Y}_\pm\in S_\pm$, $\forall\mathbf{Y}_\pm\in S_\pm$. If we put
$$ \mathbf{Y}_\pm=Y+\sharp_\gamma\flat_{B\pm g}Y,\;\;Y\in L,$$ we see that the previous condition holds iff $ \mathfrak{L} _{ \mathbf{X}}(\bar{B}\pm \bar{g})=0$, where $\bar{B}, \bar{g}$ are the extensions of $B,g$ by zero on $ \tilde{L}$.
Since the field satisfies the level matching constraint, we have local expressions
$$g=g_{ij}(x^k)dx^i\otimes dx^j,\;\;B=\frac{1}{2}B_{ij}(x^k)dx^i\wedge dx^j$$
with respect to distinguished local coordinates. The required conclusion follows from these expressions and formulas (\ref{genLlocal}). \end{proof}
\begin{example}\label{exKil} {\rm If the field satisfies the level matching constraint, then, for any $\tilde{L}$-foliated function $f$, $\partial f$ is a generalized Killing vector field. This follows from Example \ref{LieT} since (\ref{compH}) shows that the field $(g,B)$ satisfies the level matching constraint iff the metric $\mathcal{H}$ is a strongly foliated tensor field.}\end{example}
\section{Canonical connections and curvature}
In double field theory one is interested in connections on $TM$ that
preserve the metrics $\gamma$ and $\mathcal{H}$. If we see them as connections $\nabla$
on the synonymous bundle $L\oplus L^*$ to $TM$, the metric-preservation conditions (\ref{preservg}) become
\begin{equation}\label{nablagamma} \begin{array}{l} \mathbf{Z}(\gamma((X,\alpha),(Y,\beta)))
=\gamma(\nabla_{\mathbf{Z}}(X,\alpha),(Y,\beta))
+\gamma((X,\alpha),\nabla_{\mathbf{Z}}(Y,\beta)),
\vspace*{2mm}\\
\mathbf{Z}(\mathcal{H}((X,\alpha),(Y,\beta)))
=\mathcal{H}(\nabla_{\mathbf{Z}}(X,\alpha),(Y,\beta))
+\mathcal{H}((X,\alpha),\nabla_{\mathbf{Z}}(Y,\beta)),
\end{array}\end{equation}
where $\mathbf{Z}\in TM$, $X,Y\in L$ and $\alpha,\beta\in ann\,\tilde{L}$.

Equivalently, we may replace the second condition (\ref{nablagamma}) by the commutation condition $\nabla\Phi=\Phi\nabla$. This shows that $\nabla$ preserves the subbundles $S_\pm$ and must have an expression of the form
\begin{equation}\label{nablaSpm}
\nabla_{\mathbf{Z}}(X,\flat_{B\pm g}X)=(D^\pm_{\mathbf{Z}}X,\flat_{B\pm g}D^\pm_{\mathbf{Z}}X),\end{equation}
where $D^\pm$ are connections on $L$ that preserve the metric $g$.
Hence, there exists a bijective correspondence between the $(\gamma,\mathcal{H})$-preserving connections $\nabla$ on $TM$, which we will call {\it double-metric connections} hereafter, and the pairs $D^\pm$ of $g$-preserving connections on $L$.

Below, we continue to use our notation convention and alternatively denote tangent vectors of $M$ as pairs $(X,\alpha)\in L\oplus L^*$ and as vectors $\mathbf{X}=X+\sharp_\gamma\alpha$ without further warning.
The Levi-Civita connection $\nabla^0$ of $\gamma$ is not a double-metric connection since it does not preserve the metric $\mathcal{H}$. We will look for canonical, double-metric connections $\nabla$ by imposing restrictive conditions on the Gualtieri torsion of the connection.

The expression of the Gualtieri torsion of a double-metric connection is obtained as follows. We refer to the metric algebroid $(E=L\oplus L^*,\rho=Id,g=\gamma,[\,,\,]=[\,,\,]_{\nabla^0})$. Then, formula (\ref{modtors0}) gives the modified torsion of $\nabla$:
$$ T^\nabla((X,\alpha),(Y,\beta))=
\nabla_{\mathbf{X}}(Y,\beta)-
\nabla_{\mathbf{Y}}(X,\alpha)-
[(X,\alpha),(Y,\beta)]^\nabla_{\nabla^0},$$
where
$$[(X,\alpha),(Y,\beta)]_{\nabla^0}^\nabla=[(X,\alpha),(Y,\beta)]_{\nabla^0}
+(X,\alpha)\wedge_{\nabla}(Y,\beta)$$
is the modified bracket defined by (\ref{crmodD}).

The explicit formula that defines the modified bracket is
\begin{equation}\label{crmodif} \begin{array}{c}
\gamma([(X,\alpha),(Y,\beta)]_{\nabla^0}^\nabla,(Z,\zeta))=
\gamma([(X,\alpha),(Y,\beta)]_{\nabla^0},(Z,\zeta))\vspace*{2mm}\\-
\frac{1}{2}\gamma( \nabla_\mathbf{Z} (X,\alpha),(Y,\beta))
+\frac{1}{2}\gamma( \nabla_\mathbf{Z} (Y,\beta),(X,\alpha)).
\end{array}\end{equation}

We emphasize the following expression of the modified bracket defined by a double-metric connection $\nabla$:
$$
[\mathbf{X},\mathbf{Y}]_{\nabla^0}^\nabla =
[\mathbf{X},\mathbf{Y}]_{\nabla^0}+\sharp_\gamma(\sigma^\pm(X,Y)),
$$ where  $X,Y$ are the $L$-components of $\mathbf{X},\mathbf{Y}$,
$\sigma^\pm(X,Y)=0$ if $\mathbf{X}\in S_\pm,\mathbf{Y}\in S_\mp $ and $\sigma^\pm(X,Y)$ are $1$-forms defined on $M$ by the formula
$$<\sigma^\pm(X,Y),\mathbf{Z}>=\pm( g(X,D^\pm_\mathbf{Z}Y)-g(D^\pm_\mathbf{Z}X,Y)),$$
if $\mathbf{X}\in S_\pm,\mathbf{Y}\in S_\pm $.
This result is a straightforward consequence of (\ref{crmodif}) and (\ref{defHS}). The terms $\sigma^\pm$ satisfy the properties
$$ \sigma^\pm(X,Y)=-\sigma^\pm(Y,X),\;
\sigma^\pm(X,fY)=f\sigma^\pm(X,Y)\pm g(X,Y)df.$$

Using  (\ref{crmodif}), we get the expression of the Gualtieri torsion
$$\begin{array}{l}
\mathcal{T}^{\nabla}((X,\alpha),(Y,\beta),(Z,\zeta))
= \gamma(T^\nabla((X,\alpha),(Y,\beta)),(Z,\zeta))
\vspace*{2mm}\\

=\gamma(\nabla_{\mathbf{X}}(Y,\beta)-
\nabla_{\mathbf{Y}}(X,\alpha)-
[(X,\alpha),(Y,\beta)]_{\nabla^0},(Z,\zeta))
\vspace*{2mm}\\
+\frac{1}{2}\{\gamma(\nabla_{
\mathbf{Z}}(X,\alpha),(Y,\beta)) -\gamma(\nabla_{
\mathbf{Z}}(Y,\beta),(X,\alpha))\}.\end{array}
$$

Furthermore, the decomposition $L\oplus L^*=S_+\oplus S_-$ and the total skew symmetry of $\mathcal{T}^{\nabla}$ show the existence of a decomposition of the Gualtieri torsion into four components computed on arguments that belong to $S_\pm$, which we will denote by indices $\pm$, respectively. The components $\mathcal{T}^{\nabla}( \mathbf{X}_-,\mathbf{Y}_+,\mathbf{Z}_+)$,
$\mathcal{T}^{\nabla}( \mathbf{X}_+,\mathbf{Y}_-,\mathbf{Z}_-)$
will be called the mixed torsion and the components
$\mathcal{T}^{\nabla}( \mathbf{X}_+,\\
\mathbf{Y}_+,\mathbf{Z}_+)$,
$\mathcal{T}^{\nabla}( \mathbf{X}_-,\mathbf{Y}_-,\mathbf{Z}_-)$
will be called the pure torsion.
\begin{prop}\label{propconex1} For a given field $(g,B)$, there exists a unique double-metric connection $\nabla^1$ on $M$ that has a vanishing mixed torsion and its restrictions to $S_\pm$ are determined by the Levi-Civita connection of the metric $g$ of $L$.\end{prop}
\begin{proof}
The vanishing of the mixed torsion was used for a different purpose in \cite{G2}. The exact sense of the word ``determined" in the second condition will be explained in the course of the proof below.

Since $\nabla$ preserves $S_+$ and $S_-$ and these two bundles are $\gamma$-orthogonal, the condition
$\mathcal{T}^\nabla(\mathbf{X}_-,\mathbf{Y}_+,\mathbf{Z}_+)=0$ becomes
\begin{equation}\label{TG01} \gamma(\nabla_{\mathbf{X}_-}\mathbf{Y}_+,\mathbf{Z}_+)= \gamma([\mathbf{X}_-,\mathbf{Y}_+]_{\nabla^0},\mathbf{Z}_+),\end{equation}
which yields the covariant derivative $$\nabla_{ \mathbf{X}_-} \mathbf{Y}_+=pr_{S_+}[\mathbf{X}_-,\mathbf{Y}_+]_{\nabla^0}.$$
Similarly, the condition $\mathcal{T}^\nabla(\mathbf{X}_+,\mathbf{Y}_-,\mathbf{Z}_-)=0$ is equivalent to
\begin{equation}\label{TG02} \gamma(\nabla_{\mathbf{X}_+}\mathbf{Y}_-,\mathbf{Z}_-)= \gamma([\mathbf{X}_+,\mathbf{Y}_-]_{\nabla^0},\mathbf{Z}_-),\end{equation} which yields $$\nabla_{ \mathbf{X}_+} \mathbf{Y}_-=pr_{S_-}[\mathbf{X}_+,\mathbf{Y}_-]_{\nabla^0}.$$

We also notice that the conditions (\ref{TG01}), (\ref{TG02}) are equivalent to
\begin{equation}\label{DpmcuTG0} \begin{array}{l}
D^+_{\iota_{-}X}Y=pr_Lpr_{S_+}[\iota_{-}X,\iota_{+}Y]_{\nabla^0}
\hspace{3mm}(X,Y\in L)\vspace*{2mm}\\
D^-_{\iota_{+}X}Y=pr_Lpr_{S_-}[\iota_{+}X,\iota_{-}Y]_{\nabla^0}
\hspace{3mm}(X,Y\in L),\end{array}\end{equation}
where $D^\pm$ are the $g$-metric connections on $L$ that correspond to the double-metric connection $\nabla$.

Furthermore, the covariant derivatives $D^{\pm}_{(X,0)}$, $X\in L$ produce connections $D^{\pm, L}$ defined along the foliation $L$ (i.e., connections on the direct sum of the leaves of $L$ applied to arguments that are differentiable on $M$, also called partial connections) given by
$$D^{\pm, L}_XY=D^{\pm}_{\iota_+X_1}Y+D^{\pm}_{\iota_-X_2}Y\,\,(X,Y\in L),$$ where $X_1,X_2$ are given by the formula (\ref{descSpm}) with $\alpha=0$. If $\nabla$ has zero mixed Gualtieri torsion, using (\ref{DpmcuTG0}), the previous formulas become
\begin{equation}\label{DpmL}\begin{array}{l}
D^{+, L}_XY=D^{+}_{\iota_+X_1}Y+pr_Lpr_{S_+}[\iota_-X_2,\iota_+Y]_{\nabla^0},\,
\vspace*{2mm}\\
D^{-, L}_XY=D^{-}_{\iota_-X_2}Y+pr_Lpr_{S_-}[\iota_+X_1,\iota_-Y]_{\nabla^0}.
\end{array}\end{equation}
Since $g$ is symmetric and non degenerate, and $B$ is skew symmetric, it follows easily that the mappings $$A_\pm=\frac{1}{2}(Id\pm\sharp_g\flat_B):L\rightarrow L$$ are isomorphisms, which allows us to transform (\ref{DpmL}) into
\begin{equation}\label{DpmL2}\begin{array}{l} D^{+}_{\iota_+X}Y= D^{+,L}_{A_-^{-1}X}Y -pr_Lpr_{S_+}[\iota_-A_+
A_-^{-1}X,\iota_+Y]_{\nabla^0},\vspace*{2mm}\\
D^{-}_{\iota_-X}Y= D^{-,L}_{A_+^{-1}X}Y  -pr_Lpr_{V_-}[\iota_-A_-A_+^{-1}X,\iota_-Y]_{\nabla^0}.
\end{array}\end{equation}

Accordingly, there exists a unique, double-metric connection $\nabla$, with vanishing mixed Gualtieri torsion and such that the two connections $D^{\pm, L}$ are equal to the Levi-Civita connection of the metric $g$. This is the required condition of the proposition with the precise meaning of the determination of $\nabla|_{S_\pm}$ by the Levi-Civita connection of $g$. It only remains to denote $\nabla=\nabla^1$, $D^{\pm}=D^{1,\pm}$.
\end{proof}
\begin{defin}\label{defc1} {\rm The connection $\nabla^1$ provided by Proposition \ref{propconex1} will be called the {\it CWT (canonical with torsion) double-metric connection }.}\end{defin}

A second canonical, double-metric connection will be obtained by the following procedure.
\begin{prop}\label{propG0} The double-metric connection $\nabla$ of $M$ has a vanishing Gualtieri torsion iff the following relation holds for any vector fields
\begin{equation}\label{TG04} \begin{array}{l}
\gamma(\nabla_{\mathbf{X}}\mathbf{Y},\mathbf{Z})
+\gamma(\nabla_{\mathbf{Y}}\mathbf{Z},\mathbf{X})
+\gamma(\nabla_{\mathbf{Z}}\mathbf{X},\mathbf{Y})\vspace*{2mm}\\
=\gamma([\mathbf{X},\mathbf{Y}]_{\nabla^0},\mathbf{Z})
+\gamma([\mathbf{Y},\mathbf{Z}]_{\nabla^0},\mathbf{X})
+\gamma([\mathbf{X},\mathbf{Z}]_{\nabla^0},\mathbf{Y})\vspace*{2mm}\\
-\frac{1}{2}[\mathbf{X}(\gamma(\mathbf{Y},\mathbf{Z}))
-\mathbf{Y}(\gamma(\mathbf{Z},\mathbf{X}))
-3\mathbf{Z}(\gamma(\mathbf{X},\mathbf{Y}))].\end{array}\end{equation}
\end{prop}
\begin{proof}
Like in Proposition \ref{propconex1}, the vanishing of the mixed torsion implies the formulas (\ref{TG01}), (\ref{TG02}). Furthermore, we must have $\mathcal{T}^\nabla(\mathbf{X}_\pm,\mathbf{Y}_\pm,\mathbf{Z}_\pm)=0$, equivalently,
$$\begin{array}{c}\gamma(\nabla_{\mathbf{X}_\pm}\mathbf{Y}_\pm,\mathbf{Z}_\pm)
-\gamma(\nabla_{\mathbf{Y}_\pm}\mathbf{X}_\pm,\mathbf{Z}_\pm)
+\frac{1}{2}\gamma(\nabla_{\mathbf{Z}_\pm}\mathbf{X}_\pm,\mathbf{Y}_\pm)\vspace*{2mm}\\
-\frac{1}{2}\gamma(\nabla_{\mathbf{Z}_\pm}\mathbf{Y}_\pm,\mathbf{X}_\pm)
=\gamma([\mathbf{X}_\pm,\mathbf{Y}_\pm]_{\nabla^0},\mathbf{Z}_\pm).\end{array}$$
To this equality we add its first cyclic permutation, then, subtract the second cyclic permutation. The result is
$$\begin{array}{l}
\frac{3}{2}\gamma(\nabla_{\mathbf{X}_\pm}\mathbf{Y}_\pm,\mathbf{Z}_\pm)
+\frac{1}{2}\gamma(\nabla_{\mathbf{X}_\pm}\mathbf{Z}_\pm,\mathbf{Y}_\pm)
-\frac{3}{2}\gamma(\nabla_{\mathbf{Z}_\pm}\mathbf{Y}_\pm,\mathbf{X}_\pm)\vspace*{2mm}\\
-\frac{1}{2}\gamma(\nabla_{\mathbf{Z}_\pm}\mathbf{X}_\pm,\mathbf{Y}_\pm)
+\frac{1}{2}\gamma(\nabla_{\mathbf{Y}_\pm}\mathbf{Z}_\pm,\mathbf{X}_\pm)
-\frac{1}{2}\gamma(\nabla_{\mathbf{Y}_\pm}\mathbf{X}_\pm,\mathbf{Z}_\pm)\vspace*{2mm}\\
=\gamma([\mathbf{X}_\pm,\mathbf{Y}_\pm]_{\nabla^0},\mathbf{Z}_\pm)
+\gamma([\mathbf{Y}_\pm,\mathbf{Z}_\pm]_{\nabla^0},\mathbf{X}_\pm)
+\gamma([\mathbf{X}_\pm,\mathbf{Z}_\pm]_{\nabla^0},\mathbf{Y}_\pm),\end{array}$$
which, modulo (\ref{nablagamma}), becomes
\begin{equation}\label{TG03} \begin{array}{l}
\gamma(\nabla_{\mathbf{X}_\pm}\mathbf{Y}_\pm,\mathbf{Z}_\pm)
+\gamma(\nabla_{\mathbf{Y}_\pm}\mathbf{Z}_\pm,\mathbf{X}_\pm)
+\gamma(\nabla_{\mathbf{Z}_\pm}\mathbf{X}_\pm,\mathbf{Y}_\pm)\vspace*{2mm}\\
=\gamma([\mathbf{X}_\pm,\mathbf{Y}_\pm]_{\nabla^0},\mathbf{Z}_\pm)
+\gamma([\mathbf{Y}_\pm,\mathbf{Z}_\pm]_{\nabla^0},\mathbf{X}_\pm)
+\gamma([\mathbf{X}_\pm,\mathbf{Z}_\pm]_{\nabla^0},\mathbf{Y}_\pm)\vspace*{2mm}\\
-\frac{1}{2}[\mathbf{X}_\pm(\gamma(\mathbf{Y}_\pm,\mathbf{Z}_\pm))
-\mathbf{Y}_\pm(\gamma(\mathbf{Z}_\pm,\mathbf{X}_\pm))
-3\mathbf{Z}_\pm(\gamma(\mathbf{X}_\pm,\mathbf{Y}_\pm))].\end{array}\end{equation}

Formula (\ref{TG03}) is the same as (\ref{TG04}) if all the arguments are either in $S_+$ or in $S_-$.

On the other hand, if (\ref{TG04}) is written for arguments $ \mathbf{X}_\mp,
\mathbf{Y}_\pm,\mathbf{Z}_\pm$, while using $S_+\perp_{\gamma}S_-$, the result is
\begin{equation}\label{aux1}\begin{array}{c}
\gamma(\nabla_{\mathbf{X}_\mp}\mathbf{Y}_\pm,\mathbf{Z}_\pm)
=\gamma([\mathbf{X}_\mp,\mathbf{Y}_\pm]_{\nabla^0},\mathbf{Z}_\pm)
+\gamma([\mathbf{Y}_\pm,\mathbf{Z}_\pm]_{\nabla^0},\mathbf{X}_\mp)\vspace*{2mm}\\
+\gamma([\mathbf{X}_\mp,\mathbf{Z}_\pm]_{\nabla^0},\mathbf{Y}_\pm)
-\frac{1}{2}\mathbf{X}_\mp(\gamma(\mathbf{Y}_\pm,\mathbf{Z}_\pm)).
\end{array}\end{equation}

Finally, if we use (\ref{axvCalg}) for $\mathbf{X}_\mp, \mathbf{Y}_\pm,\mathbf{Z}_\pm$, we see that (\ref{aux1}) reduces to (\ref{TG01}), (\ref{TG02}), thus, justifying the general formula (\ref{TG04}).
\end{proof}

Of course, for the connections $\nabla$ with vanishing Gualtieri torsion the formulas (\ref{DpmcuTG0}) also hold.
Furthermore, the equality (\ref{TG03}) is equivalent to
\begin{equation}\label{DpmcuTG3}\begin{array}{l} g(D^\pm_{\iota_\pm X}Y,Z)
+g(D^\pm_{\iota_\pm Y}Z,X)+g(D^\pm_{\iota_\pm Z}X,Y)
\vspace*{2mm}\\ = \pm g(pr_Lpr_{S_\pm}[\iota_{\pm}{X},\iota_{\pm}{Y}]_{\nabla^0},Z)
\pm g(pr_Lpr_{S_\pm}[\iota_{\pm}{Y},\iota_{\pm}{Z}]_{\nabla^0},X)
\vspace*{2mm}\\ \pm g(pr_Lpr_{S_\pm}[\iota_{\pm}{X},\iota_{\pm}{Z}]_{\nabla^0},Y)
\mp\frac{1}{2}[(\iota_{\pm}X)(g(Y,Z))\vspace*{2mm}\\-(\iota_{\pm}Y)(g(Z,X))
-3(\iota_{\pm}Z)(g(X,Y)],\;\;X,Y,Z\in \Gamma L.\end{array}\end{equation}

Now, we continue as follows. Let $\tilde{\nabla}$ be an arbitrary double-metric connection and put
\begin{equation}\label{Theta1} \nabla_{ \mathbf{X}} \mathbf{Y}= \tilde{\nabla}_{ \mathbf{X}} \mathbf{Y}+\Theta( \mathbf{X},\mathbf{Y}),\end{equation}
where $\Theta$ is a tensor field of type $(1,2)$. We will also denote
\begin{equation}\label{Psi} \Psi( \mathbf{X},\mathbf{Y},\mathbf{Z})=
\gamma(\Theta( \mathbf{X},\mathbf{Y}),\mathbf{Z}).\end{equation} Since the two connections preserve $\gamma$, we must have
\begin{equation}\label{Psi2} \Psi( \mathbf{X},\mathbf{Y},\mathbf{Z})=
-\Psi( \mathbf{X},\mathbf{Z},\mathbf{Y}).\end{equation} We will refer to (\ref{Theta1}) as a {\it deformation} of the connection $\tilde{\nabla}$, with deformation tensor $\Theta$ and covariant deformation $\Psi$.
\begin{prop}\label{propdeform} For any double-metric connection $\tilde{\nabla}$ there exists a unique deformation with a totally skew symmetric, covariant deformation tensor that leads to a double-metric connection $\nabla$ with a vanishing Gualtieri torsion.\end{prop}
\begin{proof}
With (\ref{Theta1}), (\ref{Psi}), (\ref{Psi2}), formula (\ref{TG04}) becomes
\begin{equation}\label{TG05} \begin{array}{l}
3Alt(\Psi(\mathbf{X},\mathbf{Y},\mathbf{Z}))
=\gamma([\mathbf{X},\mathbf{Y}]_{\nabla^0},\mathbf{Z})
+\gamma([\mathbf{Y},\mathbf{Z}]_{\nabla^0},\mathbf{X})\vspace*{2mm}\\
+\gamma([\mathbf{X},\mathbf{Z}]_{\nabla^0},\mathbf{Y})
-\frac{1}{2}[\mathbf{X}(\gamma(\mathbf{Y},\mathbf{Z}))
-\mathbf{Y}(\gamma(\mathbf{Z},\mathbf{X}))\vspace*{2mm}\\
-3\mathbf{Z}(\gamma(\mathbf{X},\mathbf{Y}))]
-[\gamma(\tilde{\nabla}_{\mathbf{X}}\mathbf{Y},\mathbf{Z})
+\gamma(\tilde{\nabla}_{\mathbf{Y}}\mathbf{Z},\mathbf{X})
+\gamma(\tilde{\nabla}_{\mathbf{Z}}\mathbf{X},\mathbf{Y})],\end{array}
\end{equation}
where $Alt$ denotes the alternation of a tensor. (Notice that the symmetrization $Sym(\Psi(\mathbf{X},\mathbf{Y},\mathbf{Z}))=0$ because of (\ref{Psi2})).

The metric axiom (\ref{axvCalg}) allow us to check that the right hand side of (\ref{TG05}) is a totally skew symmetric tensor field (it is $C^\infty(M)$-trilinear). Therefore, if we define $\Psi$ by the right hand side of (\ref{TG05}), we get the required deformation.\end{proof}
\begin{rem}\label{obsunicitate} {\rm If the condition of skew symmetry of $\Psi$ is dropped, while still asking (\ref{Psi2}), we get a family of double-metric connections with a vanishing Gualtieri torsion. All the tensor fields $\Psi$ that satisfy (\ref{Psi2}) are obtained by adding to the totally skew symmetric solution of (\ref{TG05}) any $3$-covariant tensor that is skew symmetric in the last two arguments and has a vanishing alternation. Such tensors exist because their vector space is of dimension larger than that of the totally skew symmetric tensors and $Alt$ is an epimorphism.}\end{rem}

Concerning the connection deformations above we also have the following results. The preservation of the generalized metric $\mathcal{H}$ by both $\nabla$ and $\tilde{\nabla}$ implies that the operators $i(\mathbf{X})\Theta$ preserve $S_\pm$, i.e.,
$$ \Theta(\mathbf{X},\iota_\pm Y)=
\iota_\pm(\theta^\pm(\mathbf{X},Y)),\;\;Y\in L,$$
where $\theta^\pm$ are $L$-valued tensorial forms defined on $TM\times L$, which, therefore, determine $\Theta$.

Then, for arguments $\mathbf{X}=\iota_\pm X,\mathbf{Y}=\iota_\pm Y,\mathbf{Z}=\iota_\pm Z$ ($X,Y,Z\in L$) formula (\ref{TG05}) reduces to
\begin{equation}\label{Altpsi} \begin{array}{l}
3Alt(\psi^\pm(X,Y,Z))=
\pm g(pr_Lpr_{S_\pm}[\iota_{\pm}{X},\iota_{\pm}{Y}]_{\nabla^0},Z)\vspace*{2mm}\\
\pm g(pr_Lpr_{S_\pm}[\iota_{\pm}{Y},\iota_{\pm}{Z}]_{\nabla^0},X)
\pm g(pr_Lpr_{S_\pm}[\iota_{\pm}{X},\iota_{\pm}{Z}]_{\nabla^0},Y)\vspace*{2mm}\\
\mp\frac{1}{2}[(\iota_{\pm}X)(g(Y,Z))-(\iota_{\pm}Y)(g(Z,X))
-3(\iota_{\pm}Z)(g(X,Y)]\vspace*{2mm}\\ \mp
[g(\tilde{D}^{\pm}_{\iota_\pm X}Y,Z)+g(\tilde{D}^{\pm}_{\iota_\pm Y}Z,X)+g(\tilde{D}^{\pm}_{\iota_\pm Z}X,Y)],
\end{array}\end{equation}
where $\tilde{D}^{\pm}$ are the connections on $L$ that determine $\tilde{\nabla}$ in the sense of (\ref{nablaSpm}) and
$$ \psi^\pm(X,Y,Z)
=\Psi(\iota_\pm X,\iota_\pm Y,\iota_\pm Z)=g(\theta^\pm(\iota_\pm X, Y),Z).
$$

Obviously, the unique deformation $\nabla$ of Proposition \ref{propdeform} is defined by the unique couple of $3$-forms $\psi^\pm$ given by (\ref{Altpsi}).
\begin{prop}\label{thconex} For any field $(g,B)$ there exists a unique double-metric connection $\nabla$ on $M$ with the following properties: 1) the Gualtieri torsion of $\nabla$ is zero, 2) the covariant deformation $\Psi$ of $\Theta=\nabla-\nabla^1$, where $\nabla^1$ is the CWT connection of the field, is totally skew symmetric and defined by {\rm(\ref{TG05})} (equivalently, by {\rm(\ref{Altpsi})}).\end{prop}
\begin{proof} Apply the construction of Proposition \ref{propdeform} to $\tilde{\nabla}=\nabla^1$.\end{proof}
\begin{defin}\label{defconexcan} {\rm The connection $\nabla$ given by Theorem \ref{thconex} will be called the {\it VTC (vanishing torsion canonical) connection of the field $(g,B)$}.}\end{defin}

In order to produce a field theory, one needs an action expressed by an integral where the integrand is an invariant of the scalar curvature type. We start by considering the modified curvature tensor of an arbitrary double-metric connection $\nabla$, which, by (\ref{modcurvD}), has the expression
$$
\mathcal{R}^\nabla(\mathbf{X},\mathbf{Y}) \mathbf{Z}=\nabla_\mathbf{X}
\nabla_\mathbf{Y} \mathbf{Z}-\nabla_\mathbf{Y}
\nabla_\mathbf{X}
\mathbf{Z}-\nabla_{[\mathbf{X},\mathbf{Y}]_{\nabla^0}^\nabla} \mathbf{Z}.
$$

Furthermore, consider the pseudo-Kronecker symbol
$$ \delta_{(p)ij}=\delta_{(p)}^{ij}=
 \delta_{(p)i}^j=\left\{\begin{array}{rl} 0&{\rm if}\,i\neq j
 \vspace*{2mm}\\ 1&{\rm if}\,i=j=1,...,p
 \vspace*{2mm}\\ -1&{\rm if}\,i=j=p+1,...,m.
 \end{array}\right. $$
Then, we can define a modified Ricci curvature by
$$ \begin{array}{c}
\rho^\nabla( \mathbf{X},\mathbf{Y})
=\delta_{(p)}^{ij}\mathcal{H}(\iota_+e_i,
\mathcal{R}^\nabla(\mathbf{X},\iota_+e_j) \mathbf{Y})\vspace*{2mm}\\ +\delta_{(p)}^{ij}\mathcal{H}(\iota_-e_i,\mathcal{R}^\nabla( \mathbf{X},\iota_-e_j) \mathbf{Y}),\end{array}$$
where $(e_i)$ is a pseudo-orthonormal basis and $p$ is the positive inertia index of $g$. The independence of $\rho^\nabla$ upon the choice of the basis is a consequence of $\mathcal{H}|_{S_\pm}=2\iota_\pm g$.

We shall also define the symmetrized, modified Ricci tensor
$$
 \rho_{sym}^\nabla(\mathbf{X},\mathbf{Y})= \frac{1}{2}(\rho^\nabla( \mathbf{X},\mathbf{Y}) +\rho^\nabla(\mathbf{Y},\mathbf{X})).$$

Finally, we define the modified scalar curvature by
$$ \kappa( \mathcal{H},\nabla)=
\delta_{(p)}^{ij}\rho_{sym}^\nabla(\iota_+e_i,\iota_+e_j)
+\delta_{(p)}^{ij}\rho_{sym}^\nabla(\iota_-e_i,\iota_-e_j).$$

Now, the transformation formulas (\ref{locafin}) show that the double manifold $M$ is orientable and the expression
$$ d(vol_{\mathcal{H}})=\sqrt{|det( \mathcal{H})|}dx^1\wedge...\wedge dx^m\wedge d\tilde{x}_1\wedge...\wedge d\tilde{x}_m$$
is a global volume form, which is parallel with respect to any double-metric connection. Thus, it is natural to integrate with respect to this form in the definition of an action. (In physics, it is more usual to integrate densities rather than forms; see the Appendix at the end of the paper.)

Of course, from the local basis $(dx^i,d\tilde{x}_j)$ of the cotangent bundle $T^*M$ we may go to
an arbitrary basis $\theta^u$, in which case $\sqrt{|det( \mathcal{H})|}$ is multiplied by $|det(S)|^{-1}$ where $S$ is the matrix transforming $(dx^i,d\tilde{x}_j)$ into $(\theta^u)$. In particular, let $(e_i)$ be a local basis of $L$ (not necessarily pseudo-orthonormal) and let $\epsilon^j$ be the corresponding dual basis $(\epsilon^j(e_i)=\delta_i^j)$. Then, we have the basis $(\iota_+e_i,\iota_-e_i)$ in $TM$ and its dual basis $\epsilon^{+,j}=\iota_+^{-1*}\epsilon^{j},
\epsilon^{-,j}=\iota_-^{-1*}\epsilon^{j}$. The relation (\ref{defHS}) between the metrics $g$ and $\mathcal{H}$ has the consequence that the value of $\sqrt{|det( \mathcal{H})|}$ in these bases is equal to $2^m|det(g)|$, which yields
$$ d(vol_{\mathcal{H}})=2^m|det(g)|\epsilon^{+,1}\wedge...\wedge \epsilon^{+,m}\wedge \epsilon^{-,1}\wedge...\wedge \epsilon^{-,m}.$$

Accordingly, and also taking into consideration the scalar dilation $\varphi$ of the field, we may look at two canonically defined actions:
$$ \mathcal{A}=\int_Me^{-2\varphi}\kappa(\mathcal{H},\nabla)d(vol_{\mathcal{H}}),\;
\mathcal{A}_1=\int_Me^{-2\varphi}\kappa(\mathcal{H},\nabla^1)d(vol_{\mathcal{H}}),
$$
where $\nabla$ is the VCT connection and $\nabla^1$ is the CWT connection of $M$. Of course, we will have to impose conditions ensuring that the integrals are finite.

The study of these actions and their possible interest for physics is beyond the scope of this paper.

We end this section with the following remarks. The modified curvature of a double-metric connection with a vanishing Gualtieri torsion satisfies the Bianchi identity
\begin{equation}\label{Bianchi} \sum_{Cycl(\mathbf{X},\mathbf{Y},\mathbf{Z})} \mathcal{R}^\nabla(\mathbf{X},\mathbf{Y})\mathbf{Z}=
\sum_{Cycl(\mathbf{X},\mathbf{Y},\mathbf{Z})}[\mathbf{X},
[\mathbf{Y},\mathbf{Z}]^\nabla_{\nabla^0}]^\nabla_{\nabla^0}.
\end{equation}
Formula (\ref{Bianchi}) follows by a straightforward calculation if the modified bracket that enters in the definition of $\mathcal{R}^\nabla$ is replaced by means of the vanishing torsion condition written in the form
$$ \nabla_{\mathbf{X}}\mathbf{Y}
-\nabla_{\mathbf{Y}}\mathbf{X}=[\mathbf{X},\mathbf{Y}]_{\nabla^0}^\nabla.
$$

In particular, if either $\mathbf{X},\mathbf{Y},\mathbf{Z}\in L$ or $\mathbf{X},\mathbf{Y},\mathbf{Z}\in \tilde{L}$, then,
$$\sum_{Cycl(\mathbf{X},\mathbf{Y},\mathbf{Z})} \mathcal{R}^\nabla(\mathbf{X},\mathbf{Y})\mathbf{Z}=0.$$
Indeed, we may compute the value of the left hand side of (\ref{Bianchi}) on arguments $\mathbf{X}_x,\mathbf{Y}_x,\mathbf{Z}_x$ at a point $x\in M$ by extending the arguments to local, strongly foliated vector fields in $L,\tilde{L}$, respectively, and by computing the value of the right hand side of (\ref{Bianchi}) at $x$. Then, (\ref{Jptcr}) holds and shows that the result is zero.
\section{Para-Dirac structures on double manifolds}
The topic of this section was not considered in double field theory and we study it only from the point of view of geometry. For any vector space or vector bundle $ \mathcal{S}$ endowed with a non degenerate, neutral metric $\gamma$, a maximal $\gamma$-isotropic subspace $ \mathcal{D}$ is called a Dirac subspace or subbundle, respectively. In particular, if $ \mathcal{S}=TM$ where $(M,\gamma)$ is a flat, para-K\"ahler manifold, then, $ \mathcal{D}$ will be called an {\it almost para-Dirac structure}\footnote{We add the particle ``para" to avoid confusion with the usual notion of a Dirac structure where $\mathcal{S}=TM\oplus T^*M$ \cite{C}.} on $M$, and if $\Gamma\mathcal{D}$ is closed under the metric $\nabla^0$-bracket, we will call it an {\it integrable} or a {\it para-Dirac structure}.

Some of the known algebraic facts concerning almost Dirac structures \cite{C} also apply to the almost para-Dirac case. We may consider the field of tangent subspaces $ \mathcal{E}=pr_L\mathcal{D}\subseteq TM$ and the $2$-form $\varpi$ induced on $\mathcal{E}$ by the fundamental form $\omega$ (see Section 1):
$$\varpi(X,Y)=\omega(\mathbf{X},\mathbf{Y}),\;\; X=pr_L\mathbf{X},Y=pr_L\mathbf{Y},\;\mathbf{X},\mathbf{Y}\in\mathcal{D},$$ which is independent of the choice of the extensions $\mathbf{X},\mathbf{Y}$ because $\mathcal{D}$ is $\gamma$-isotropic.
Then, we have the following reconstruction of $\mathcal{D}$:
\begin{equation}\label{reconstr} \mathcal{D}=\{ \mathbf{X}\in TM\,/\, pr_L\mathbf{X}\in \mathcal{E},\;\varpi(pr_L\mathbf{X},Y) =\gamma(pr_{\tilde{L}}\mathbf{X},Y),\,\forall Y\in \mathcal{E}\},\end{equation}
where everything is at the points $x\in M$. Indeed, if $\mathbf{X},\mathbf{X}'$ are of the form described by (\ref{reconstr}), then, $\gamma(\mathbf{X},\mathbf{X}')=0$ is a consequence of the skew symmetry of $\varpi$, and (\ref{reconstr}) implies $dim\,\mathcal{D}=m$.

On the other hand, following our earlier work \cite{V2}, we prove
\begin{prop}\label{Dsiisometr} Let $M$ be the double manifold of a field $(g,B)$. Then, there exists a bijective correspondence between the almost para-Dirac structures $ \mathcal{D}$ on $M$ that have a non degenerate restriction of the corresponding generalized metric $ \mathcal{H}$ and the tensor fields $J\in\Gamma(L\otimes L^*)$ that are isometries of the metric $g$. \end{prop}
\begin{proof} We start with the structure $\mathcal{D}$. Since $\mathcal{H}|_{\mathcal{D}}$ is non degenerate, $\mathcal{D}\cap\mathcal{D}^{\perp_{ \mathcal{H}}}=0$, hence $TM=\mathcal{D}\oplus\mathcal{D}^{\perp_\mathcal{H}}$. This decomposition is an almost product structure with a tensor field $\Psi$ such that $\mathcal{D},\mathcal{D}^{\perp_{ \mathcal{H}}}$ are the $\pm1$-eigenbundles of $\psi$. Checking for all the possible combinations of arguments in
$\mathcal{D},\mathcal{D}^{\perp_{ \mathcal{H}}}$, we get the compatibility condition
\begin{equation}\label{compatPsi} \mathcal{H}(\Psi\mathbf{X},\Psi\mathbf{Y})=\mathcal{H}( \mathbf{X},\mathbf{Y}).\end{equation}

Furthermore, with (\ref{Hgamma}) and since $\Phi$ is a $\gamma$-isometry, we get $\Phi( \mathcal{D})=\mathcal{D}^{\perp_{ \mathcal{H}}}$ and $\mathcal{D}^{\perp_{ \mathcal{H}}}$ is again maximal and $\gamma$-isotropic. The $\gamma$-isotropy of $\mathcal{D}$ and $\mathcal{D}^{\perp_{ \mathcal{H}}}$ yields the compatibility condition $$\gamma(\Psi\mathbf{X},\mathbf{Y})=- \gamma(\mathbf{X},\Psi\mathbf{Y}), $$
equivalently, $\Phi\circ\Psi=-\Psi\circ\Phi$.

The last relation implies $\Psi(S_\pm)=S_\mp$. Hence, there must exist a bundle isomorphism $J:L\rightarrow L$ such that, if we use $L\oplus L^*$ instead of $TM$, we have
\begin{equation}\label{eqPsi}\begin{array}{l} \Psi(X,\flat_{B+ g}X)=(JX,\flat_{B- g}JX),\vspace*{2mm}\\
\Psi(X,\flat_{B- g}X)=(J^{-1}X,\flat_{B+ g}J^{-1}X),\;\;X\in L.\end{array}\end{equation}
Moreover, the compatibility condition (\ref{compatPsi}) translates into
$$ g(JX,JY)=g(X,Y),\,g(J^{-1}X,J^{-1}Y)=g(X,Y),\;\;\forall X,Y\in L,
$$ which means that $J$ is a $g$-isometry.

Conversely, if we start with the $g$-isometry $J$ of $L$, it is easy to check that \begin{equation}\label{DdinJ} \mathcal{D}=\{(X,\flat_{B+ g}X) +(JX,\flat_{B- g}JX)\,/\,X\in L\}\end{equation} is a maximal, $\gamma$-isotropic subbundle, i.e., an almost para-Dirac structure on $M$. Its $ \mathcal{H}$-orthogonal bundle is
$$\mathcal{D}^{\perp_{ \mathcal{H}}}=\{(X,\flat_{B+ g}X) -(JX,\flat_{B- g}JX)\,/\,X\in L\}$$ and, since it has the intersection $0$ with $\mathcal{D}$, $\mathcal{H}|_{\mathcal{D}}$ is non degenerate. Moreover, it follows that the corresponding $g$-isometry is exactly the initial isometry $J$.\end{proof}

In analogy with the almost Dirac case, some almost para-Dirac structures may be interpreted as double objects of $2$-forms and bivector fields on the bundle $L$.
\begin{prop}\label{paraPoisson} An almost para-Dirac structure $ \mathcal{D}$ such that $ \mathcal{D}\cap\tilde{L}=0$ is equivalent with a $2$-form $\theta\in\wedge^2L^*$. An almost para-Dirac structure $ \mathcal{D}$ such that $ \mathcal{D}\cap L=0$ is equivalent with a bivector field $P\in\wedge^2L$.\end{prop}
\begin{proof} $ \mathcal{D}\cap\tilde{L}=0$ implies $TM= \mathcal{D}\oplus\tilde{L}$, hence, $\forall\tilde{X}\in L$ we have a corresponding decomposition $$X=X'+X'',\;X'\in\mathcal{D},X''\in\tilde{L}.$$
If $X\in L$, this yields the decomposition of $X'$ when $TM=L\oplus \tilde{L}$ and we see that $pr_L\mathcal{D}=L$. Furthermore, $\mathcal{D}\cap\tilde{L}=0$ also shows that the vector of $\mathcal{D}$ that projects to a given $X\in L$ is unique. Together with the $\gamma$-isotropy of $ \mathcal{D}$, the previous remarks lead to the conclusion that $\mathcal{D}=graph\,\flat_\theta$ for a well defined form $\theta\in\wedge^2L^*$. This proves the first assertion of the proposition. If the roles of $L$ and $\tilde{L}$ are interchanged in the previous argument, we get $\mathcal{D}=graph\,\sharp_P$ where $P\in\wedge^2L$, which justifies the second assertion.
\end{proof}

In the first case of Proposition \ref{paraPoisson}, if we ask $graph\,\flat_\theta$ to be of the form (\ref{DdinJ}), the comparison yields
the corresponding $g$-isometry
$$ J_\theta=(Id+\sharp_g\flat_{B-\theta})\circ
(Id-\sharp_g\flat_{B-\theta})^{-1}.$$ In the case of $graph\,\sharp_P$ we get similarly
$$ J_P=(Q^+-Id)\circ(Q^-+Id)^{-1},\;
Q^\pm=\pm\sharp_g\flat_{B\pm g}\sharp_P\flat_g.$$ In both cases, the existence of the required inverses is ensured by the non degeneracy of $g$.
\begin{rem}\label{obsparaD} {\rm
The general almost para-Dirac structures $ \mathcal{D}$ can be related to objects on $L$ in the same way as for the almost Dirac structures. We indicate this briefly and refer to \cite{IV} for details. Using $L\oplus L^*$ instead of $TM$, the $ \mathcal{H}$-compatible, almost product structure $\Psi$ that corresponds to $ \mathcal{D}$ may be written as
\begin{equation}\label{matriceaPsi} \Psi\left(
\begin{array}{c}X\vspace{2mm}\\ \alpha \end{array}
\right) = \left(\begin{array}{cc} A&\sharp_\pi\vspace{2mm}\\
\flat_\sigma&-^t\hspace{-1pt}A\end{array}\right)
\left( \begin{array}{c}X\vspace{2mm}\\
\alpha \end{array}\right),\end{equation}
where $A\in End (L),\pi\in\Gamma\wedge^2L,\sigma\in\Gamma\wedge^2L^*$ and
$$ A^2=Id -
\sharp_\pi\circ\flat_\sigma,\;\pi(\alpha\circ A,\beta)=\pi(\alpha,
\beta\circ A),\;\sigma(AX,Y)=\sigma(X,AY).$$ Furthermore, the triple $(A,\sigma,\pi)$ must be such that the compatibility condition $\Phi\circ\Psi=-\Psi\circ\Phi$ holds. The relation between the triple $(A,\sigma,\pi)$ and the tensor $J$ of Proposition \ref{Dsiisometr} follows by comparing the representations (\ref{eqPsi}) and (\ref{matriceaPsi}) of $\Psi$. The result is \cite{V2}:
$$ \begin{array}{ll}
J=A+\sharp_\pi\circ\flat_{B+g},&
\flat_{B-g}\circ J=\flat_\sigma-
\hspace{1pt}^t\hspace{-1pt}A\circ\flat_{B+g},\vspace*{2mm}\\
J^{-1}=A+\sharp_\pi\circ\flat_{B-g},&
\flat_{B+g}\circ J^{-1}=\flat_\sigma-
\hspace{1pt}^t\hspace{-1pt}A\circ\flat_{B-g},\end{array}
$$
$$\begin{array}{c}
\sharp_\pi=\frac{1}{2}(J-J^{-1})\circ\sharp_g,\;
A=\frac{1}{2}(J+J^{-1})-\sharp_\pi\flat_B,\vspace*{2mm}\\
\flat_\sigma=\flat_B\circ(J+J^{-1})-\hspace{1pt}^t\hspace{-1pt}A\circ\flat_B.
\end{array}$$}\end{rem}

Now, we look at the para-integrability condition .
\begin{prop}\label{paraintc} The almost para-Dirac structure $\mathcal{D}$ is integrable iff, $\forall\mathbf{X},\mathbf{Y},\mathbf{Z}\in\Gamma \mathcal{D}$, one of the following equivalent conditions holds:\\
$1)$\hspace{5mm} $\gamma([ \mathbf{X},\mathbf{Y}]_{\nabla^0},\mathbf{Z})=0$,\\
$2)$\hspace{5mm} $\gamma(\mathbf{X},\nabla^0_{\mathbf{Z}}\mathbf{Y})=
\gamma([\mathbf{X},\mathbf{Y}],\mathbf{Z})$,\\
$3)$\hspace{5mm} $\sum_{Cycl(\mathbf{X},\mathbf{Y},\mathbf{Z})}
\gamma(\mathbf{X}, \nabla^0_{\mathbf{Z}}\mathbf{Y})=0$,\\
$4)$\hspace{5mm} $\sum_{Cycl(\mathbf{X},\mathbf{Y},\mathbf{Z})}
\gamma([\mathbf{X},\mathbf{Y}], \mathbf{Z})=0$,\\
$5)$\hspace{5mm} $\sum_{Cycl(\mathbf{X},\mathbf{Y},\mathbf{Z})}
\gamma(\mathbf{X}\wedge_{\nabla^0}\mathbf{Y}, \mathbf{Z})=0$.
\end{prop}
\begin{proof}
Condition 1) is equivalent to para-integrability because $\mathcal{D}$ is $\gamma$-maximally isotropic.
Using formula (\ref{Ccroset}) and the consequence
$\gamma(\mathbf{X}, \nabla^0_{\mathbf{Z}}\mathbf{Y})=-
\gamma(\nabla^0_{\mathbf{Z}}\mathbf{X}, \mathbf{Y})$ of $\gamma(\mathbf{X},\mathbf{Y})=0$, condition 1) transforms into 2).
Then, since $\nabla^0$ has no torsion condition 2) is equivalent to 3).
If we use the well known global expression of the Levi-Civita connection (e.g., \cite{KN}) in 3), we get condition 4), and the latter, together with formula (\ref{equax}) yields condition 5).
\end{proof}

The para-integrability condition 2)  has the following obvious consequence.
\begin{prop}\label{prop123} If $ \mathcal{D}$ is an almost para-Dirac structure, any two of the following properties implies the third property: $1)$
$ \mathcal{D}$ is para-integrable, $2)$ $ \mathcal{D}$ is a foliation on $M$, $3)$ $ \mathcal{D}$ is totally geodesic (i.e., $\nabla^0_{\mathbf{X}}\mathbf{Y}\in\Gamma\mathcal{D}$, $\forall
\mathbf{X},\mathbf{Y}\in\Gamma \mathcal{D}$).\end{prop}

An interesting situation is that of a {\it strongly foliated} para-Dirac structure $ \mathcal{D}$, which is defined by the condition that the subbundle
$ \mathcal{D}\subseteq TM$ has local bases consisting of strongly foliated, local vector fields. In this case, Corollary \ref{algCfol} implies that $ \mathcal{D}$ is an $\tilde{L}$-transversal Lie algebroid \cite{VMedJ}. On the other hand, we get a nice form of the para-integrability condition, similar to that of Dirac structures:
\begin{prop} A strongly foliated almost para-Dirac structure is integrable iff \begin{equation}\label{paraint5} d_{\nabla^0}\omega(\mathbf{X},\mathbf{Y},\mathbf{Z})=0,\;\; \forall\mathbf{X},\mathbf{Y},\mathbf{Z}\in\Gamma \mathcal{D},\end{equation}
where $d_{\nabla^0}$ is defined by replacing the Lie brackets by $\nabla^0$-brackets in the formula of the exterior differential $d$.
\end{prop}
\begin{proof} Since it suffices to check (\ref{paraint5}) at the points $x\in M$ and since the hypothesis that $\mathcal{D}$ is strongly foliated implies that each vector at $x$ has extensions to strongly foliated, local vector fields in $\mathcal{D}$, it suffices to prove the result for strongly foliated arguments. Then, $\mathbf{X}\wedge_{\nabla^0}\mathbf{Y}\in\Gamma\tilde{L}$ and the para-integrability condition 5) is equivalent to
$$\sum_{Cycl(\mathbf{X},\mathbf{Y},\mathbf{Z})}
\omega(\mathbf{X}\wedge_{\nabla^0}\mathbf{Y}, \mathbf{Z})=0,\;\; \forall\mathbf{X},\mathbf{Y},\mathbf{Z}\in\Gamma \mathcal{D}.$$
Here, if we replace $$\mathbf{X}\wedge_{\nabla^0}\mathbf{Y}\stackrel{(\ref{Ccroset})}{=}
[\mathbf{X},\mathbf{Y}]- [\mathbf{X},\mathbf{Y}]_{\nabla^0}$$ and use the consequence
$$\sum_{Cycl(\mathbf{X},\mathbf{Y},\mathbf{Z})}
\omega([\mathbf{X},\mathbf{Y}], \mathbf{Z})= \sum_{Cycl(\mathbf{X},\mathbf{Y},\mathbf{Z})} \mathbf{X}(\omega(\mathbf{Y},\mathbf{Z}))$$ of the property $d\omega=0$, we get (\ref{paraint5}).\end{proof}

The particular case of strongly foliated para-Dirac structures of the form $\mathcal{D}=graph\,\flat_\theta$, $\mathcal{D}=graph\,\sharp_P$ $(\theta\in\Gamma\wedge^2L^*,P\in\Gamma\wedge^2L)$ occurs iff $\theta,P$ are strongly foliated tensor fields. Then, the structure is the lift of local, usual Dirac structures of the corresponding type (i.e., presymplectic and Poisson, respectively) on local, transversal submanifolds of the foliation $\tilde{L}$ and it is well known that the integrability conditions of the latter are $d_L\theta=0$ and the annulation of the Schouten-Nijenhuis bracket $[P,P]_L=0$, respectively.
\section{Appendix: Densities on vector bundles}
The local transition functions of a vector bundle structure yield a bijective correspondence between the isomorphism classes of vector bundles $V$ of rank $k$ and the cohomology classes in $H^1(M,\underline{Gl(k,\mathds{R})})$ (e.g., \cite{V6}); underlining denotes the sheaf of germs of $C^\infty$-functions with values in the group). Then, $\forall s\in\mathds{R}$, there exists a well defined homomorphism $H^1(M,\underline{Gl(k,\mathds{R})})
\rightarrow H^1(M,\underline{\mathds{R}_+})$ ($\mathds{R}_+$ is the multiplicative group of positive real numbers) defined by $$(\psi^i_j)\mapsto |det(\psi^i_j)|^s,$$ where $(\psi^i_j)$ is the matrix that transforms old coordinates into new coordinates, $\xi'^i=\psi^i_j\xi^j$. A line bundle that corresponds to the image cohomology class is called a bundle of densities of weight $s$ of $V$; this bundle is only determined up to an isomorphism and it is trivial since the sheaf $\underline{\mathds{R}}_+$ is fine, therefore, it has trivial cohomology in positive dimensions.

If the bundle $V$ is orientable, the canonical bundle $\wedge^kV$ is a bundle of densities of rank $1$. For this reason, in the general case, if $(e_i)$ is a local basis of cross sections of $V$, we will denote by $|e_1\wedge...\wedge e_k|^s$ a corresponding local basis of the bundle of densities of weight $s$. \begin{example}\label{exdens} {\rm
It $g$ is a metric on the vector bundle $V$, the function $\sqrt{|det(g)|}$ is the component of a density of weight $-1$ of the bundle $V$, equivalently, of a density of weight $1$ of the bundle $V^*$.}\end{example}

If $V$ is an arbitrary vector bundle of rank $k$ endowed with a connection $\nabla$ that has the local equations
$$\nabla e_i=\omega_i^je_j$$ where $\omega_i^j$ are the local connection forms, then, $\wedge^kV$ has an induced connection with the local equation $$\nabla(e_1\wedge...\wedge e_k)=\varpi(e_1\wedge...\wedge e_k),\hspace{2mm} \varpi=trace\,\omega,\,\omega=(\omega_i^j).$$

Under a transition $ \tilde{e}_j=\lambda_j^ie_i$, where $\lambda_j^i\psi_i^k=\delta_j^k$, the connection form $\varpi$ changes by
 $$ \tilde{\varpi}=\varpi+d\,ln|det(\lambda_j^i)|.$$This equality may also be written as
$$ \tilde{\varpi}=\varpi+\frac{1}{s}d\,ln|det(\lambda_j^i)^s|.$$ Therefore, the equation $$\nabla|e_1\wedge...\wedge e_k|^s=s\varpi|e_1\wedge...\wedge e_k|^s$$ defines an induced connection on the bundle of densities of weight $s$.

If $V=TM$ where $M$ is a flat, para-K\"ahler manifold, we may consider
a strongly foliated density of weight $s$ of the tangent bundle $TM$, $$\theta=\vartheta(x^i)\left|\frac{\partial}{\partial x^1}\wedge...\wedge
\frac{\partial}{\partial x^m}\wedge\frac{\partial}{\partial \tilde{x}_1}
\wedge...\wedge\frac{\partial}{\partial \tilde{x}_m}\right|^s,$$ where $(x^i,\tilde{x}_j)$ are distinguished coordinates. Then, we define the generalized Lie derivative
\begin{equation}\label{deriLdens} \mathfrak{L}_{ \mathbf{X}}\theta
=(\mathfrak{L}_{ \mathbf{X}}\vartheta+s\vartheta div_L\mathbf{X})
\left|\frac{\partial}{\partial x^1}\wedge...\wedge
\frac{\partial}{\partial x^m}\wedge\frac{\partial}{\partial \tilde{x}_1}
\wedge...\wedge\frac{\partial}{\partial \tilde{x}_m}\right|^s,
\end{equation} where $ \mathbf{X}$ is the vector field (\ref{folvf}) and the divergence is defined by
$$div_L\mathbf{X}=\sum_{i=1}^m\frac{\partial\xi^i}{\partial x^i}$$ and it is invariant under the coordinate transformations (\ref{locafin}).

\hspace*{7.5cm}{\small \begin{tabular}{l} Department of
Mathematics\\ University of Haifa, Israel\\ E-mail: vaisman@math.haifa.ac.il
\end{tabular}}
\end{document}